\documentclass[lettersize,journal]{IEEEtran}
\bstctlcite{bstctl:etal, bstctl:nodash, bstctl:simpurl}

\IEEEoverridecommandlockouts                              % This command is only
\usepackage{amsmath,graphicx,amsfonts,amssymb,amsthm,epsfig,mathrsfs,balance}
\usepackage{subcaption,tikz,color,multirow,algorithm,algpseudocode,algorithmicx}
\usepackage[noadjust]{cite}
\usepackage{url,framed,bm,dsfont} 
\usepackage{nicefrac}
\newtheorem{theorem}{Theorem}

\newtheorem{definition}{Definition}

\newtheorem{remark}{Remark}

\usepackage{enumitem}  

\newcommand{\tr}{{\mathrm{trace}}}

\newcommand{\differential}{{\rm{d}}}

\newcommand{\blue}{\color{black}}

\title{\LARGE \bf
Neural Schr\"{o}dinger Bridge with Sinkhorn Losses: Application to Data-driven Minimum Effort Control of  Colloidal Self-assembly
}

\author{Iman Nodozi,~\IEEEmembership{Student Member,~IEEE}, Charlie Yan, Mira Khare, Abhishek Halder,~\IEEEmembership{Senior Member,~IEEE}, Ali Mesbah,~\IEEEmembership{Senior Member,~IEEE}% <-this % stops a space
\thanks{Iman Nodozi and Charlie Yan are with the Department of Electrical and Computer Engineering, University of California, Santa Cruz, CA 95064, USA, {\texttt{\{inodozi,cyan140\}@ucsc.edu}}.}
\thanks{Abhishek Halder (corresponding author) is with the Department of Aerospace Engineering, Iowa State University, Ames, IA 50011, USA, {\texttt{ahalder@iastate.edu}}.} 
\thanks{Mira Khare and Ali Mesbah are with the Department of Chemical and Biomolecular Engineering, University of California, Berkeley, CA 94720, USA, {\tt\small{\{mira\_khare,mesbah\}@berkeley.edu}}.}        
\thanks{This research was supported by NSF awards 2112754 and 2112755.}
}

\begin{document}
% \bstctlcite{TCST-SBP-SA:BSTcontrol}
\bstctlcite{IEEE_b:BSTcontrol}
% \bstctlcite{BSTcontrol}
\maketitle
% change plain -> empty before submission
\thispagestyle{plain}
\pagestyle{plain}

%%%%%%%%%%%%%%%%%%%%%%%%%%%%%%%%%%%%%%%%%%%%%%%%%%%%%%%%%%%%%%%%%%%%%%%%%%%%%%%%
\begin{abstract}
We show that the minimum effort control of colloidal self-assembly can be naturally formulated in the order-parameter space as a generalized Schr\"{o}dinger bridge problem -- a class of fixed-horizon stochastic optimal control problems that originated in the works of Erwin Schr\"{o}dinger in the early 1930s. In recent years, this class of problems has seen a resurgence of research activities in the control and machine learning communities. Different from the existing literature on the theory and computation for such problems, the controlled drift and diffusion coefficients for colloidal self-assembly are typically nonaffine in control, and are difficult to obtain from physics-based modeling. We deduce the conditions of optimality for such generalized problems, and show that the resulting system of equations is structurally very different from the existing results in a way that standard computational approaches no longer apply. Thus motivated, we propose a data-driven learning and control framework, named `neural Schr\"{o}dinger bridge', to solve such generalized Schr\"{o}dinger bridge problems by innovating on recent advances in neural networks. We illustrate the effectiveness of the proposed framework using a numerical case study of colloidal self-assembly. We learn the controlled drift and diffusion coefficients as two neural networks using molecular dynamics simulation data, and then use these two to train a third network with Sinkhorn losses designed for distributional endpoint constraints, specific for this class of control problems.   
\end{abstract}

%%%%%%%%%%%%%%%%%%%%%%%%%%%%%%%%%%%%%%%%%%%%%%%%%%%%%%%%%%%%%%%%%%%%%%%%%%%%%%%%

\noindent{\bf Keywords: Schr\"{o}dinger bridge, Sinkhorn loss, stochastic optimal control, physics-informed neural networks, colloidal self-assembly.}

%%%%%%%%%%%%%%%%%%%%%%%%%%%%%%%%%%%%%%%%%%%%%%%%%%%%%%%%%%%%%%%%%%%%%%%%%%%%%%%%
\section{Introduction}\label{SecIntro}
Motivated by feedback control of colloidal self-assembly (SA), this work focuses on learning the solution of the nonlinear stochastic optimal control problems over a given fixed time horizon $[0, T]$ of the form
\begin{subequations}
\begin{align}
\underset{\bm{u}\in\mathcal{U}}{\inf}\:~~~~~~&\mathbb{E}_{\mu^{\bm{u}}}\left[\int_{0}^{T}\frac{1}{2}\|\bm{u}(t,\bm{x})\|_{2}^{2}\:\differential t\right]\label{genSBPobj}\\
\text{subject to} ~~~&\differential \bm{x} =\bm{f}(t,\bm{x},\bm{u}) \differential t + \sqrt{2}\: \bm{g}(t,\bm{x},\bm{u}) \differential \bm{w},\label{GenSBPdyn}\\
 & \bm{x}(t=0)\sim \mu_{0}\;\text{(given)}, \; \bm{x}(t=T)\sim \mu_{T}\;\text{(given)},\label{genSBPconstr}    \end{align}
\label{GeneralizedSBP}
\end{subequations}
where $\mu_{0},\mu_{T}$ denote the prescribed probability measures over the state space $\mathcal{X}\subseteq\mathbb{R}^{n}$ at $t=0$ and $t=T$, respectively. The constraint in \eqref{GenSBPdyn} is a controlled It\^{o} stochastic differential equation (SDE) with the state vector $\bm{x}\in \mathcal{X}$, the control vector $\bm{u}\in\mathbb{R}^{m}$, and the standard Wiener process $\bm{w}\in\mathbb{R}^{p}$. For the solution to the SDE (1b) to be 
for colloidal SA systems, the state vector $\bm{x}$ represents suitable order parameters. The \emph{drift coefficient} $\bm{f}$ is a vector field given by mapping $\bm{f}:[0,T]\times\mathcal{X}\times\mathcal{U} \mapsto \mathbb{R}^{n}$, and the \emph{diffusion coefficient} $\bm{g}$ is a matrix field given by mapping $\bm{g}:[0,T]\times\mathcal{X}\times\mathcal{U} \mapsto \mathbb{R}^{n\times p}$. For the SDE solutions to be well-posed, we will detail suitable smoothness assumptions on $\bm{f}$ and $\bm{g}$. 

Associated with the diffusion coefficient $\bm{g}$, is a \emph{diffusion tensor} $\bm{G}:=\bm{g}\bm{g}^{\top}\in\mathbb{S}^{n}_{+}$, which being an outer product, is a symmetric positive semidefinite matrix field $\bm{G}:[0,T]\times\mathcal{X}\times\mathcal{U} \mapsto\mathbb{S}^{n}_{+}$. In \eqref{genSBPobj}, we suppose that the set of admissible controls $\mathcal{U}$ comprises of finite energy Markovian inputs within a prescribed time horizon, i.e., 
\begin{align}
  \mathcal{U}:=\left\{\bm{u}: [0,T] \times \mathcal{X}\mapsto \mathbb{R}^{m} \mid\langle\bm{u}, \bm{u}\rangle<\infty\right\},\label{feasibleU} 
\end{align}
where $\langle\cdot, \cdot\rangle$ denotes the standard Euclidean inner product. The symbol $\mathbb{E}_{\mu^{\bm{u}}}\left[\cdot\right]$ in \eqref{genSBPobj} denotes the mathematical expectation w.r.t. the controlled state probability measure $\mu^{\bm{u}}$, that is, $\mathbb{E}_{\mu^{\bm{u}}}\left[\cdot\right] := \int (\cdot)\:\differential\mu^{\bm{u}}$. The superscript $\bm{u}$ in $\mu^{\bm{u}}$ indicates that the joint measure depends on the choice of control $\bm{u}$. Thus, the objective in \eqref{genSBPobj} is to minimize the control effort in steering the state statistics from $\mu_0$ to $\mu_T$ under a prespecified time horizon and controlled stochastic dynamics constraints, over all admissible control policies $\bm{u}(t,\bm{x})$ in $\mathcal{U}$.

In feedback control of colloidal SA systems, the objective generally is to design control policies that steer the system from an initial disordered stochastic state to a desired terminal ordered crystalline stochastic state \cite{PAULSON201538,tang2022control}. These stochastic states are naturally encoded in terms of suitable order parameters. As such,  formulation \eqref{GeneralizedSBP} is particularly appealing in this context because it allows for directly shaping the multivariate distribution of order parameters via optimal control synthesis. The drift and diffusion coefficients $\bm{f},\bm{g}$ in equation \eqref{GenSBPdyn} allow for the representation of the free energy landscape, which is crucial for circumventing kinetic traps or local minima when directing the system towards a desired end state, typically a global minimum within the solution space.

However, in practice, the drift and diffusion coefficients are difficult to model from first principles. This is because accurately capturing the interplay between various forces and interactions, such as van der Waals forces, electrostatic interactions, and solvent-mediated interactions, is challenging. As a result, empirical or semi-empirical approaches \cite{tang2013colloidal}, as well as coarse-grained or phenomenological models \cite{tang2016optimal}, are often employed to approximate these coefficients based on either experimental data or molecular dynamics (MD) simulation data.

Another modeling difficulty specific to colloidal SA is that both $\bm{f},\bm{g}$ are typically nonlinear in state $\bm{x}$, as well as non-affine in control $\bm{u}$. Furthermore, $\bm{f}$ and $\bm{g}$ both have explicit time dependence in practice. To circumvent these modeling issues, in this work, we propose a learning and control framework where $\bm{f}$ and $\bm{g}$ are learnt from high-fidelity MD simulation data as the outputs of neural network (NN) representations $\mathcal{N}_{\text {Drift}}$ and $\mathcal{N}_{\text {Diffusion}}$, respectively. With these learnt representations for $\bm{f}$ and $\bm{g}$, we propose a computational framework--based on another neural network--to numerically solve \eqref{GeneralizedSBP} for control synthesis.

\subsection{Relation to the Schr\"{o}dinger Bridge Problem (SBP)}\label{subsecRelSBP}
We refer to \eqref{GeneralizedSBP} as a \emph{generalized Schr\"{o}dinger Bridge Problem (GSBP)} since it is related to distributional two-point boundary value problems originating in two papers of Erwin Schr\"{o}dinger in 1931-32 \cite{schrodinger1931umkehrung,schrodinger1932theorie}. The qualifier ``generalized" points to the presence of prior dynamics given by the controlled drift-diffusion coefficient pair $\left(\bm{f},\bm{g}\right)$, which generalizes the setting considered in Schr\"{o}dinger's original investigations \cite{schrodinger1931umkehrung,schrodinger1932theorie}. In the special case $\bm{f}\equiv \bm{u}$, $\bm{g}\equiv \bm{I}_{n}$, formulation \eqref{GeneralizedSBP} reduces to the classical SBP \`{a} la Schr\"{o}dinger. From this perspective, classical SBP is the problem of minimum effort distribution steering with zero prior drift, i.e., the problem of controlling Brownian motion with endpoint distribution constraints. 

A different way to interpret classical SBP is to view it as a stochastic dynamic version of the optimal mass transport (OMT) problem. The dynamic OMT \cite{benamou2000computational} is a special case of \eqref{GeneralizedSBP} with $\bm{f}\equiv \bm{u}$, $\bm{g}\equiv \bm{0}$. For details on these connections from a stochastic control perspective, we refer the readers to \cite{chen2016relation}. In recent years, SBPs and their generalizations have come to prominence in both control \cite{chen2016relation,caluya2021wasserstein,caluya2021reflected,chen2021stochastic} and machine learning \cite{de2021diffusion,wang2021deep,chen2021likelihood,croitoru2023diffusion} communities. In particular, a data-driven maximum likelihood sampling solution of the classical SBP (i.e., with $\bm{f}\equiv \bm{u}$, $\bm{g}\equiv \bm{I}_{n}$) was proposed in \cite{pavon2021data}  assuming availability of the samples from the endpoint measures $\mu_{0},\mu_{T}$. Similar line of ideas were pursued in \cite{vargas2021solving,stromme2023sampling}.

While solution methods for the GSBP \eqref{GeneralizedSBP} in general are not available in the current literature, specialized algorithms for particular forms of $\bm{f}, \bm{g}$ have appeared. For instance, \cite{caluya2021wasserstein} considered the case when the drift coefficient $\bm{f}$ is control affine and the diffusion coefficient $\bm{g}$ is $C\left([0,T]\right)$ matrix that is independent of state and input, i.e., $$\bm{f}\left(t,\bm{x},\bm{u}\right) \equiv \widetilde{\bm{f}}\left(t,\bm{x}\right) + \bm{B}(t)\bm{u}, \:\bm{g}\left(t,\bm{x},\bm{u}\right) \equiv \bm{B}(t)\in\mathbb{R}^{n\times m}.$$
In this case, $m=p$ and the stochastic process noise enters through the input channels (e.g., modeling disturbance in forcing and/or actuation uncertainties). The results in \cite{caluya2021wasserstein} showed that if $\widetilde{\bm{f}}$ is gradient of a potential, or if $(\widetilde{\bm{f}},\bm{B}(t))$ is of mixed conservative-dissipative form, then certain proximal recursions can be designed to numerically solve the corresponding GSBP. In \cite{caluya2021reflected}, this result was extended for the case when additional (deterministic) state constraints are present. 

GSBPs with nonlinear drifts and full-state feedback linearizable structures were considered in \cite{caluya2020finite,haddad2020prediction}.  GSBP instances for both first- and second-order noisy nonuniform Kuramoto oscillator models were solved in \cite{nodozi2022schrodinger} using Feynman-Kac path integral techniques. Closest to the GSBP \eqref{GeneralizedSBP} is the work in \cite{nodozi2022physics}, which considered \emph{control non-affine} drift and diffusion coefficients and showed that the conditions of optimality involves additional coupled PDEs compared to the control-affine case. However, the developments in \cite{nodozi2022physics} were still model based. Data-driven solution of control non-affine GSBPs at the level of generality \eqref{GeneralizedSBP}, as pursued in this work, is novel w.r.t. the existing literature.

\subsection{Related Works on Control of Colloidal Self-Assembly}\label{SubsecRelatedWorks}

Feedback control has emerged as a promising approach to enhance the reproducibility of colloidal SA systems towards desired structures. Previous studies \cite{juarez2012feedback, gao2018feedback} demonstrated the effectiveness of proportional-integral control on simple test systems. However, applying such basic control approaches to complex colloidal SA systems with possible kinetically arrested dynamics may not yield satisfactory results.  Alternative approaches like model predictive control (MPC) or dynamic programming have been suggested. For instance, \cite{tang2013colloidal} presents an MPC approach based on energy landscapes estimated from MD simulations. However, this method can be computationally demanding for large-scale systems or systems with complex interactions, especially considering that the solution time for MPC might exceed the sampling time of SA, particularly for fast dynamics. This challenge becomes even more compounded as the size and complexity of the SA model grows.

On the other hand, \cite{tang2016optimal} utilizes a dynamic programming-based approach, which results in a lookup table of optimal actions for given states. Despite its theoretical elegance, dynamic programming suffers from the `curse of dimensionality', rendering it impractical for systems of higher complexity due to the exponential growth in computational resources required. For both these methods, the accuracy of the control relies heavily on the quality of the underlying model. To this end, model-free reinforcement learning can alleviate modeling challenges in optimal control of colloidal SA systems \cite{whitelam2020learning,o2023novelty}. Furthermore, recent advances in NNs have provided a promising alternative for modeling the hidden physics of stochastic dynamic systems without making assumptions about the final equations representing the physics of the system (e.g., \cite{dijkstra2021predictive,O_Leary_2022,mao2022deep,lizano2023convolutional}). We employ the NNs, $\mathcal{N}_{\text {Drift}}$ and $\mathcal{N}_{\text {Diffusion}}$, to represent the energy and diffusion landscapes of a colloidal SA control problem, critical to guiding the system to a desired final state. 
%This perspective distinguishes our approach from the preceding methods in that they predominantly rely on predetermined equations representing the hidden physics of the system. 
Our focus shifts to governing the temporal progression of the joint probability distribution encompassing the states involved in colloidal SA. Building upon previous work \cite{nodozi2022physics}, we incorporate data-driven %high-dimensional 
models based on physics-informed neural networks and trained on high-fidelity MD simulation data, which allows for a more realistic %and assumption-free 
representation of the colloidal SA process.

% \subsubsection*{Notations}
% We use  For vectors $\bm{x}, \bm{y} \in \mathcal{X}$, we have $\langle\bm{x}, \bm{y}\rangle:=\bm{x}^{\top} \bm{y}$, and that the squared Euclidean 2-norm $\|\bm{x}\|_{2}^{2}:=\langle\bm{x}, \bm{x}\rangle$. The symbol $\sim$ is used as a shorthand for ``follows the probability distribution". We use $\mathbb{S}^{n}_{+}$ to denote the set of symmetric positive semidefinite matrices.

%%%%%%%%%%%%%%%%%%%%%%%%%%%%%%%%%%%%%%%%%%%%%%%%%

\begin{figure*}[t]
\centering
\includegraphics[width=.75\linewidth]{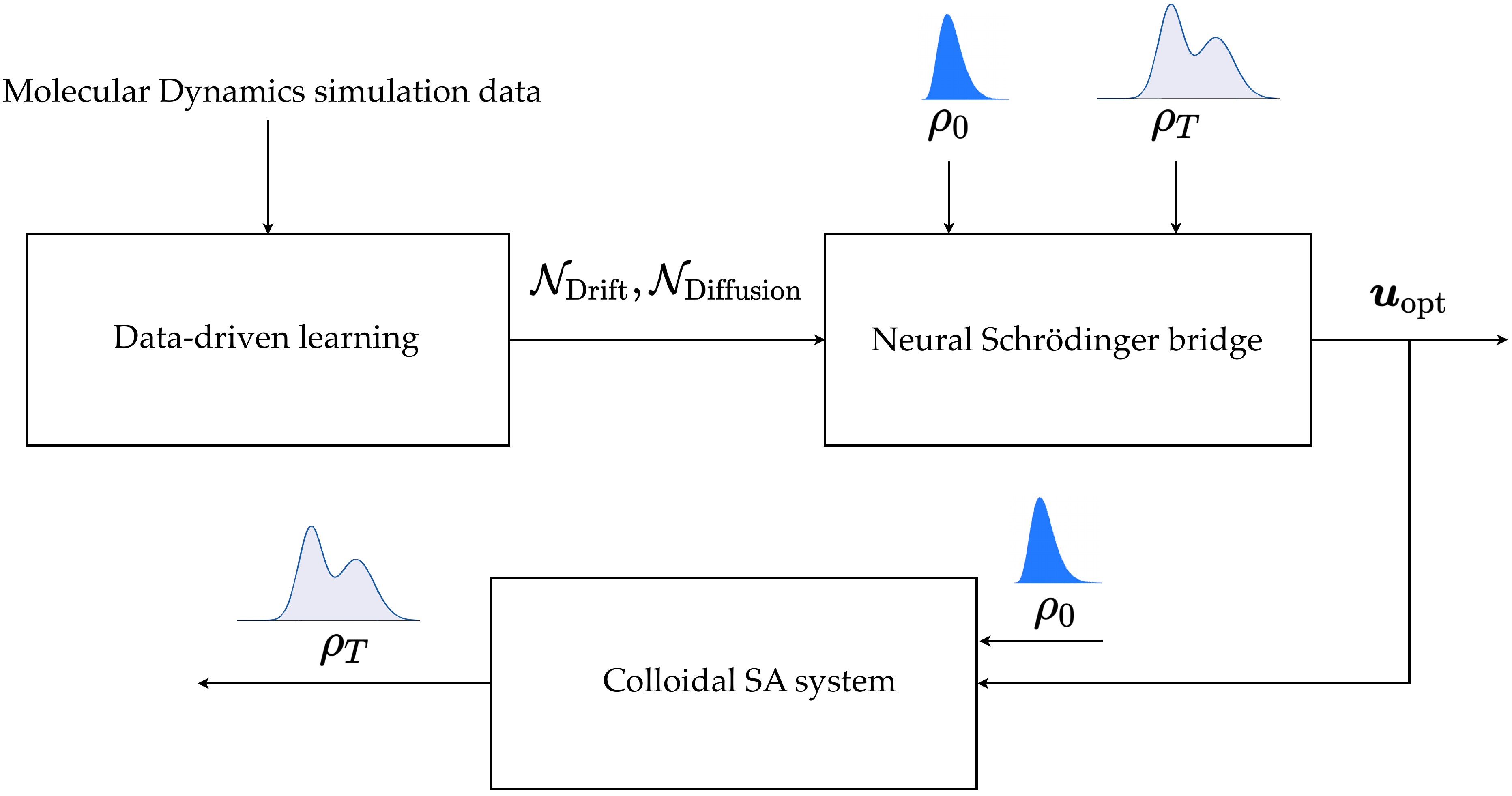}
\caption{{\small{An overview of the proposed learning and control framework for solving the generalized Schr\"{o}dinger Bridge Problem \eqref{GeneralizedSBP} for colloidal self-assembly. Here, $\rho_0$ and $\rho_{T}$ denote the probability density functions associated with the endpoint measures $\mu_0$ and $\mu_T$, respectively.}}}
\label{fig:overallTCST}
\end{figure*}

%%%%%%%%%%%%%%%%%%%%%%%%%%%%%%%%%%

\subsection{Contributions}
This paper makes the following specific contributions.
\begin{itemize}
    \item Building on our prior work \cite{nodozi2022physics}, we propose that the controlled colloidal SA can be naturally formulated as a distribution steering problem in a suitable order-parameter space. This offers a newfound connection between the controlled colloidal SA and a non-standard stochastic optimal control problem with hard constraints on the endpoint state statistics. The resulting stochastic optimal control problem takes the form of a control non-affine GSBP. 

    \item To the best of the authors' knowledge, this is the first work to derive and numerically solve the conditions of optimality for control-non-affine GSBP in multi-dimensional state-space. As detailed in Sec. \ref{SecNeuralSBP}, the resulting system of equations is fundamentally different from the control-affine SBPs in that the optimal control is no longer an explicit functional of the (sub)gradient of the associated value function solving a Hamilton-Jacobi-Bellman (HJB) PDE. Instead, the $m$-dimensional optimal control $\bm{u}_{\rm{opt}}$ here solves a system of $m$ PDEs, which are coupled nonlinearly with two more PDEs and endpoint boundary conditions. As a result, existing approaches from the literature such as the Hopf-Cole transform \cite{hopf1950partial,cole1951quasi} followed by a contractive fixed point recursion \cite{chen2016entropic,caluya2021wasserstein}, or Feynman-Kac path integral techniques \cite{nodozi2022schrodinger}, cannot be used to numerically solve our system of equations. Leveraging recent advances in NNs, we propose a computational framework to learn the solution for this system of $m+2$ PDEs and boundary conditions.

    \item Our proposed computational approach, dubbed `neural Schr\"{o}dinger bridge', is `neural' in two ways: (i) a pair of NNs are trained to approximate the $\bm{f}$ and $\bm{g}$ in \eqref{GeneralizedSBP} using MD simulation data, (ii) the GSBP optimality conditions, derived as functions of these NN representations, are further solved via a physics-informed neural network (PINN) \cite{raissi2019physics,lu2021deepxde}. However, standard PINNs with mean squared error (MSE) losses are not appropriate to enforce distributional endpoint constraints \eqref{genSBPconstr}. To address this, we propose a PINN with Sinkhorn a.k.a. entropy-regularized Wasserstein losses for these constraints, and differentiate through these losses for training. The resulting architecture could be of independent interest.   
\end{itemize}

We clarify here that, from a methodological viewpoint, the proposed framework is different from two recent works \cite{koshizuka2022neural} and \cite{kim2023unpaired}, which also bring together SBPs and NNs. In \cite{koshizuka2022neural}, the main idea was to learn the \emph{uncontrolled} $\bm{f},\bm{g}$ as NNs, i.e., to learn an \emph{unforced} neural SDE using the population samples via SBP. The unforced SDE was learnt via a stochastic version of the principle of least action, i.e., by appealing to how SBP can be seen as a stochastic version of dynamic OMT, as we explained in Sec. \ref{subsecRelSBP}. The work in \cite{kim2023unpaired} proposed learning a classical SBP between unpaired images. Different from these works, our colloidal SA context requires learning the \emph{controlled} $\bm{f},\bm{g}$ as controlled neural SDEs before proceeding for optimal control synthesis -- the latter is an instance of GSBP, which is then solved via a new variant of PINN that we propose herein.

%%%%%%%%%%%%%%%%%%%%%%%%%%%%%%%%%%

\subsection{Organization}
In Sec. \ref{SecNeuralSBP}, %we detail our overall approach of learning of the drift and diffusion coefficients $(\bm{f},\bm{g})$ as neural network representations from MD simulation data (i.e., as controlled neural SDEs), 
we define the GSBP \eqref{GeneralizedSBP} in terms of NN representations of $\bm{f}$ and $\bm{g}$. In Sec. \ref{secPINN}, we then discuss the solution of the GSBP conditions of optimality using a novel PINN with Sinkhorn losses. A detailed numerical case study of a colloidal SA system in an isothermal-isobaric (NPT) ensemble is then presented in Sec. \ref{SecNumerical}. Sec. \ref{SecConclusions} concludes the paper.

%%%%%%%%%%%%%%%%%%%%%%%%%%%%%%%%%%%%%%%%%

\section{Neural Schr\"{o}dinger Bridge}\label{SecNeuralSBP}
Our overall approach is to learn $(\bm{f},\bm{g})$ as fully connected feed-forward NN representations, denoted by $\mathcal{N}_{\text {Drift}}$ and $\mathcal{N}_{\text {Diffusion}}$, respectively. Both these NNs are designed to be functions of the current time $t\in[0,T]$, the system state $\bm{x}$, and the control input $\bm{u}$. These two networks are trained to predict the future states of the system based on the tuple $(t,\bm{x},\bm{u})$. Training of these networks using MD simulations is detailed in Sec. \ref{SecNumerical}. Fig. \ref{fig:overallTCST} gives an overview of the proposed learning and control framework. We next state the smoothness of the learnt NN representations required for the control problem (i.e., the GSBP) to be well-posed. 

\subsection{Smoothness of the Learnt $\bm{f}$ and $\bm{g}$}\label{subsecSmoothness}
We consider both $\mathcal{N}_{\text {Drift}}$ and $\mathcal{N}_{\text {Diffusion}}$ to have tangent hyperbolic, i.e., $\tanh(\cdot)$ activation functions. Tangent hyperbolic nonlinearities are known to be slope-restricted  \cite[Proposition 2]{fazlyab2020safety}. As a result, the output of a fully connected feed-forward NN with $\tanh$ activation remains component-wise slope-restricted. Consequently, $\bm{f},\bm{g}$ being the respective outputs of the networks $\mathcal{N}_{\text {Drift}},\mathcal{N}_{\text {Diffusion}}$, are guaranteed \cite[Theorem 2]{fazlyab2019efficient} to be Lipschitz continuous.

Motivated by the Lipschitz continuity of the outputs of $\mathcal{N}_{\text {Drift}},\mathcal{N}_{\text {Diffusion}}$ for an admissible Markovian policy $\bm{u}(t,\bm{x}) \in \mathcal{U}$, we assume that the coefficients $\bm{f}$ and $\bm{g}$ satisfy\\
\noindent\textbf{(A1) non-explosion and Lipschitz conditions:} there exist constants $c_1, c_2$ such that $$\|\bm{f}(t,\bm{x},\bm{u}(t,\bm{x}))\|_{2} + \|\bm{g}(t,\bm{x},\bm{u}(t,\bm{x}))\|_2 \leq c_1\left(1 + \|\bm{x}\|_2\right),$$ 
and that $$\|\bm{f}(t,\bm{x},\bm{u}(t,\bm{x})) - \bm{f}(t,\widetilde{\bm{x}},\bm{u}(t,\widetilde{\bm{x}}))\|_{2} \leq c_2\|\bm{x}-\widetilde{\bm{x}}\|_2$$ for all $\bm{x},\widetilde{\bm{x}}\in \mathcal{X}$, $t\in[0,T]$; \\
\noindent\textbf{(A2) uniformly lower bounded diffusion:}  there exists constant $c_3$ such that  the diffusion tensor $\bm{G}=\bm{g} \bm{g}^{\top}$ satisfies $$\langle \bm{x},\bm{G}\left(t,\bm{x},\bm{u}(t,\bm{x})\right)\bm{x}\rangle \geq c_3 \|\bm{x}\|_2^2$$ for all $t\in[0,T]$.\\
\noindent The assumption {\blue{\textbf{(A1)}}} guarantees \cite[p. 66]{oksendal2013stochastic} existence-uniqueness for the sample path of the SDE \eqref{GenSBPdyn}.
The assumptions {\blue{\textbf{(A1), (A2)}}} together guarantee \cite[Ch. 1]{friedman2008partial} that the generator associated with \eqref{GenSBPdyn} yields absolutely continuous probability measures $\mu^{\bm{u}}$ for all $t>0$ provided the prescribed initial probability  measure $\mu_0:=\mu^{\bm{u}}(t=0,\bm{x})$ is absolutely continuous. 

In this work, we assume that the given endpoint measures $\mu_0,\mu_T$ are absolutely continuous, i.e., $\mu_0=\rho_0(\bm{x})\differential\bm{x}$, $\mu_T = \rho_{T}(\bm{x})\differential\bm{x}$ where $\rho_0,\rho_{T}$ are the corresponding endpoint joint state probability density functions (PDFs). If the solution for \eqref{GeneralizedSBP} exists, then under the stated regularity assumptions on $\bm{f}$ and $\bm{g}$, the corresponding controlled measure $\mu^{\bm{u}}(t,\bm{x})$ will remain absolutely continuous with $\differential \mu^{\bm{u}}(t,\bm{x})=\rho^{\bm{u}}(t,\bm{x}) \differential \bm{x}$ for admissible $\bm{u}\in\mathcal{U}$. We next discuss reformulating \eqref{GeneralizedSBP} in terms of the controlled joint state PDF $\rho^{\bm{u}}$.

\subsection{PDF Steering Problem}\label{subsec:PDFsteeringproblem}
{\blue{To state the PDF steering problem, we set up some notations. We use the symbol nabla ($\nabla$) to denote the gradient w.r.t. its subscript vector. So for $\bm{x}\in\mathbb{R}^{n}$, we have 
\begin{align*}
&\text{gradient operator}\; &\nabla_{\bm{x}} &:=\begin{pmatrix}
\frac{\partial}{\partial x_1}\\
\vdots\\
\frac{\partial}{\partial x_n}
\end{pmatrix},\\
&\text{divergence operator}\;&\nabla_{\bm{x}}\cdot &:= \frac{\partial}{\partial x_1} + \hdots + \frac{\partial}{\partial x_n},\\
&\text{Laplacian operator}\; &\Delta_{\bm{x}} &:= \nabla_{\bm{x}}\cdot\nabla_{\bm{x}} = \frac{\partial^2}{\partial x_1^2} + \hdots + \frac{\partial^2}{\partial x_n^2}.
\end{align*}
For matrices $\bm{P},\bm{Q}$ with commensurate dimensions and respective $(i,j)$th entries $P_{ij},Q_{ij}$, their Frobenius a.k.a. Hibert-Schmidt inner product 
\begin{align}
\langle \bm{P},\bm{Q}\rangle := \tr\left(\bm{P}^{\top}\bm{Q}\right) = \sum_{i,j}P_{ij}Q_{ij}.
\label{defFrobInnerProduct}
\end{align}

We use the symbol $\mathbf{Hess}$ to denote the Euclidean Hessian operator defined for any real-valued twice differentiable function $h:\mathcal{X}\subseteq\mathbb{R}^{n}\mapsto\mathbb{R}$, as
$$\mathbf{Hess}\left(h\right)\!:=
\!\!\begin{pmatrix}
\frac{\partial^2 h}{\partial x_1^2} & \frac{\partial^2 f}{\partial x_1 \partial x_2} & \cdots & \frac{\partial^2 h}{\partial x_1 \partial x_n} \\
\frac{\partial^2 h}{\partial x_2 \partial x_1} & \frac{\partial^2 h}{\partial x_2^2} & \cdots & \frac{\partial^2 h}{\partial x_2 \partial x_n} \\
\vdots & \vdots & \ddots & \vdots \\
\frac{\partial^2 h}{\partial x_n \partial x_1} & \frac{\partial^2 h}{\partial x_n \partial x_2} & \cdots & \frac{\partial^2 h}{\partial x_n^2}
\!\end{pmatrix}\forall\bm{x}\in\mathcal{X}.$$
In general, the entries of $\mathbf{Hess}\left(h\right)$ depend on $\bm{x}$, i.e., $\mathbf{Hess}$ returns a symmetric matrix field. The $(i,j)$th entry of the operator $\mathbf{Hess}$ is $\frac{\partial^2}{\partial x_i\partial x_j}$.

Following \eqref{defFrobInnerProduct}, we define the Frobenius a.k.a. Hilbert-Schmidt inner product between the operator $\mathbf{Hess}$ and a matrix field $\bm{Q}(\bm{x})$ where $\bm{x}\in\mathcal{X}\subseteq\mathbb{R}^{n}$, as
\begin{align}
\langle\mathbf{Hess},\bm{Q}(\bm{x})\rangle := \sum_{i,j}\dfrac{\partial^2}{\partial x_i\partial x_j}Q_{ij}(\bm{x}).
\label{FrobInnerProductHessOperatorAndMatrix}
\end{align}
}}

With the assumptions {\blue{in Sec. \ref{subsecSmoothness}}}, the GSBP \eqref{GeneralizedSBP} can be rewritten as a state PDF steering problem:
\begin{subequations}
	\begin{align}		&\underset{(\rho^{\bm{u}},\bm{u})}{\inf}\:\int_{0}^{T}\int_{\mathcal{X}}\frac{1}{2}\|\bm{u}(t,\bm{x})\|_{2}^{2} \:\rho^{\bm{u}}(t,\bm{x}) \:\differential\bm{x}\: \differential t \label{SBPobjective} \\
		 \text{subject to} ~~&	\frac{\partial \rho^{\bm{u}}}{\partial t} =-\nabla_{\bm{x}}\cdot( \rho^{\bm{u}} \bm{f})+\langle \mathbf{H e s s},\bm{G}\rho^{\bm{u}}\rangle,\label{FPKequation}\\
	 &  \rho^{\bm{u}}(0,\bm{x})=\rho_0, \; \quad \rho^{\bm{u}}(T,\bm{x})=\rho_T \label{InitialAndTerminalPDF}.
	\end{align}
	\label{Optimization_FPK}
\end{subequations}
The constraint \eqref{FPKequation} is the controlled Fokker-Planck-Kolmogorov (FPK) PDE which governs the flow of the joint state PDF $\rho^{\bm{u}}$ associated with the SDE \eqref{GenSBPdyn}. {\blue{For a derivation of \eqref{FPKequation} from \eqref{GenSBPdyn}, see e.g., \cite[Prop. 3.3]{pavliotis2016stochastic}.

For the term $\langle \mathbf{H e s s},\bm{G}\rho^{\bm{u}}\rangle$ in \eqref{FPKequation}, note from \eqref{FrobInnerProductHessOperatorAndMatrix} that  
$$\langle \mathbf{H e s s},\bm{G}\rho^{\bm{u}}\rangle = \displaystyle\sum_{i,j}\dfrac{\partial^2}{\partial x_i\partial x_j}\left(G_{ij}(t,\bm{x},\bm{u}(t,\bm{x}))\rho^{\bm{u}}(t,\bm{x})\right).$$
}}

The boundary conditions \eqref{InitialAndTerminalPDF} at $t=0$ and $t=T$ involve the prescribed initial and terminal joint state PDFs $\rho_0$ and $\rho_T$, respectively.

We note that when $\bm{f}\equiv\bm{u}$, $\bm{g}$ (and hence $\bm{G}$) $\equiv\bm{I}_{n}$, then \eqref{FPKequation} reduces to the {\blue{controlled}} heat PDE, and problem \eqref{Optimization_FPK} reduces to the classical SBP, as mentioned in Sec. \ref{subsecRelSBP}. Furthermore, when $\bm{f}\equiv\bm{u}$, $\bm{g}\equiv\bm{0}$, then \eqref{FPKequation} reduces to the Liouville PDE \cite{halder2011dispersion} for integrator dynamics $\dot{\bm{x}}=\bm{u}$, and problem \eqref{Optimization_FPK} reduces to the dynamic OMT, as mentioned in Sec. \ref{subsecRelSBP}. 

{\blue{
\begin{remark}\label{Remark:formulationOptimalControl}
To better understand the correspondence between \eqref{GeneralizedSBP} and \eqref{Optimization_FPK}, notice that \eqref{SBPobjective} is simply a re-writing of \eqref{genSBPobj} by ``opening up" the expectation operator w.r.t. the controlled state probability measure $\differential\mu^{\bm{u}}(t,\bm{x})=\rho^{\bm{u}}(t,\bm{x})\differential\bm{x}$. The constraint \eqref{FPKequation} is the PDF dynamics induced by the sample path dynamics \eqref{GenSBPdyn}. Intuitively, the term $\frac{1}{2}\|\bm{u}\|_2^2\rho^{\bm{u}}\differential\bm{x}$ is a generalized kinetic energy, and the state-time integral \eqref{SBPobjective} encodes total control effort over the finite horizon $[0,T]$. The FPK PDE \eqref{FPKequation} is a continuity equation expressing the conservation of probability mass under the drift coefficient $\bm{f}$ and the diffusion coefficient $\bm{g}$ (thus the diffusion tensor $\bm{G}$).  
\end{remark}
}}

%%%%%%%%%%%%%%%%%%%%%%%

\subsection{Existence and Uniqueness of Solution}\label{SubsecExistenceUniqueness}
Under the assumptions stated already in Sec. \ref{subsecSmoothness}, the controlled PDF $\rho^{\bm{u}}$ exists for $\bm{u}\in\mathcal{U}$. For the existence-uniqueness of solution for the variational problem \eqref{Optimization_FPK}, we further assume that\\
\textbf{(A3)} the PDF $\rho^{\bm{u}}$ remains positive and continuous for all $t\in[0,T]$.\\ 
Then, following \cite[Thm. 3.2]{jamison1974reciprocal}, \cite{blaquiere1992controllability}, problem \eqref{Optimization_FPK} is guaranteed to admit a unique solution; see also \cite[Sec. 10]{wakolbinger1992}.

We next deduce the first order optimality conditions for the GSBP \eqref{Optimization_FPK} in the form of a coupled system of $m+2$ PDEs with boundary conditions, where $m$ is the number of control inputs. {\blue{With respect to the existing literature on the conditions of optimality for GSBPs, this system of PDEs for non-affine control is the most general, and is a new result.}}

%%%%%%%%%%%%%%%%%%%%%%%

\subsection{Conditions for Optimality}\label{SubsecOptimality}
We start with the Lagrangian associated with the GSBP \eqref{Optimization_FPK}:
\begin{equation}
\begin{aligned}
        \mathcal{L}(\rho^{\bm{u}}, \bm{u},& \psi):=\int_{0}^{T} \int_{\mathcal{X}}\left\{\frac{1}{2}\|\bm{u}(t,\bm{x})\|_{2}^{2} \rho^{\bm{u}}(t,\bm{x})+\psi(t,\bm{x}) \times\right.\\
        &\left(\frac{\partial \rho^{\bm{u}}}{\partial t}+\nabla_{\bm{x}} .( \rho^{\bm{u}} \bm{f})-\langle \mathbf{Hess},\bm{G}\rho^{\bm{u}}\rangle\right)\bigg\} \mathrm{d}\bm{x}\:\mathrm{d}t
\end{aligned}
\label{lagrangian}
\end{equation}
where $\psi(t,\bm{x})$ is a $C^{2}([0,T];\mathcal{X})$ Lagrange multiplier. Let 
\begin{align}
\mathcal{P}_{0T}(\mathcal{X}) := \big\{&\rho(t,\bm{x})\geq 0\mid \int_{\mathcal{X}}\rho\differential\bm{x} = 1 ,\nonumber\\
&\rho(t=0,\bm{x})=\rho_0,\quad \rho(t=T,\bm{x})=\rho_T\big\}.
\label{FeasiblePDFvaluedCurves}    
\end{align}
Performing the unconstrained minimization of the Lagrangian $\mathcal{L}$ over $\mathcal{P}_{0T}(\mathcal{X}) \times \mathcal{U}$, where $\mathcal{U}$ is given in \eqref{feasibleU}, we get the following result.

\begin{theorem}\label{theorem1}
    \textbf{(Optimal control and optimal state PDF)} \\
Let the set of feasible Markovian controls be given by \eqref{feasibleU}. Then the pair $(\rho^{\bm{u}}_{\rm{opt}}(t,\bm{x}),\bm{u}_{\rm{opt}}(t,\bm{x}))$ that solves \eqref{Optimization_FPK}, must satisfy the following system of $m+2$ coupled PDEs:
\begin{subequations}
	\begin{align}
		&\frac{\partial \psi}{\partial t}=\frac{1}{2} \|\bm{u}_{\rm{opt}}\|_{2}^{2}-\langle\nabla_{\bm{x}} \psi, \bm{f}\rangle-\langle \bm{G}, \mathbf{Hess}(\psi)\rangle, \label{HJB} \\
		 &	\frac{\partial \rho^{\bm{u}}_{\rm{opt}}}{\partial t} =-\nabla_{\bm{x}} \cdot( \rho^{\bm{u}}_{\rm{opt}} \bm{f})+\langle \mathbf{H e s s}, \bm{G}\rho^{\bm{u}}_{\rm{opt}}\rangle ,\label{FPK}\\
   &  \bm{u}_{\rm{opt}}=\nabla_{\bm{u}_{\rm{opt}}}\left(\langle\nabla_{\bm{x}} \psi,  \bm{f}\rangle+\langle \bm{G}, \mathbf{Hess}(\psi)\rangle\right),
    \label{OptimalInput}
	\end{align}
	\label{coupled_PDEs}
\end{subequations}
with boundary conditions
\begin{equation}
    \rho^{\bm{u}}_{\rm{opt}}(0,\bm{x})=\rho_{0}, \; \quad \rho^{\bm{u}}_{\rm{opt}}(T,\bm{x})=\rho_{T},
    \label{boundary_conditions}
\end{equation}
where $\psi(t,\bm{x})$ is a $C^{2}([0,T];\mathcal{X})$ value function.
\end{theorem}

\begin{proof}
{\blue{For $\mathcal{X}\subseteq\mathbb{R}^{n}$, let $r_{0}\in \overline{\mathbb{R}}:=\mathbb{R}\cup\{-\infty,+\infty\}$ (two point compactification of the real line $\mathbb{R}$) be defined as $r_{0}:=\sup_{\bm{x}\in\mathcal{X}}\|\bm{x}\|_2$.}} 

We write the Lagrangian \eqref{lagrangian} as the sum of three state-time integrals:
\begin{align}
&\int_{0}^{T}\!\!\!\!\int_{\mathcal{X}}\frac{1}{2}\|\bm{u}\|_{2}^{2}\rho^{\bm{u}}\differential \bm{x}\:\differential t + \int_{0}^{T}\!\!\!\!\int_{\mathcal{X}}\psi\frac{\partial\rho^{\bm{u}}}{\partial t} \differential \bm{x}\:\differential t\nonumber\\
&\!\!\!\!+\!\!\int_{0}^{T}\!\!\!\!\int_{\mathcal{X}}\!\!\left(\!\frac{\partial \rho^{\bm{u}}}{\partial t}+\nabla_{\bm{x}}\cdot( \rho^{\bm{u}} \bm{f})-\langle \bm{G}, \mathbf{Hess}\left(\rho^{\bm{u}}\!\right)\rangle\!\!\right)\!\psi\differential \bm{x}\:\differential t.
\label{ThreeTermLagrangian}
\end{align}
In above, for the second summand, we invoke the Fubini–Tonelli theorem to switch the order of integration and perform integration by parts w.r.t. $t$. {\blue{This gives
\begin{align}
&\!\!\int_{0}^{T}\!\!\!\!\int_{\mathcal{X}}\!\psi\frac{\partial\rho^{\bm{u}}}{\partial t} \differential \bm{x}\:\differential t\nonumber\\
=&\!\!\int_{\mathcal{X}}\!\!\left(\int_{0}^{T}\!\!\psi\frac{\partial\rho^{\bm{u}}}{\partial t}\differential t\right)\differential \bm{x}\nonumber\\
=&\!\!\int_{\mathcal{X}}\!\!\left(\left[\psi\rho^{\bm{u}}\right]_{t=0}^{t=T} - \int_{0}^{T}\frac{\partial\psi}{\partial t}\rho^{\bm{u}}\differential t\right)\differential \bm{x}\nonumber\\
=& \!\!\underbrace{\int_{\mathcal{X}}\!\!\left(\psi(T,\bm{x})\rho_{T}(\bm{x}) - \psi(0,\bm{x})\rho_{0}(\bm{x})\right)\differential\bm{x}}_{\text{constant over $\mathcal{P}_{0T}(\mathcal{X})\times\mathcal{U}$}} -\!\!\int_{0}^{T}\!\!\!\!\int_{\mathcal{X}}\frac{\partial\psi}{\partial t}\rho^{\bm{u}}\differential\bm{x}\differential t. 
\label{SecondSummandIBP}
\end{align}
}}

For the third summand {\blue{in \eqref{ThreeTermLagrangian}}}, we perform integration by parts w.r.t. $\bm{x}$, {\blue{to obtain
\begin{align}
&\!\!\int_{0}^{T}\!\!\!\!\int_{\mathcal{X}}\!\left(\frac{\partial \rho^{\bm{u}}}{\partial t}+\nabla_{\bm{x}}\cdot( \rho^{\bm{u}} \bm{f})-\langle \mathbf{Hess},\bm{G}\rho^{\bm{u}}\rangle\right) \psi\:\differential \bm{x}\:\differential t\nonumber\\
=&\!\!\int_{0}^{T}\!\!\!\!\left(\int_{\mathcal{X}}\!\!\left(\!\frac{\partial \rho^{\bm{u}}}{\partial t}\!+\!\nabla_{\bm{x}}\cdot( \rho^{\bm{u}} \bm{f})\!\right)\!\psi\:\differential \bm{x}\!-\!\langle \mathbf{Hess},\bm{G}\rho^{\bm{u}}\rangle\psi\differential \bm{x}\!\!\right)\!\differential t\nonumber\\
=&\!\!\!\int_{0}^{T}\!\!\!\left(\!\lim_{\|\bm{x}\|_2\rightarrow r_0}\left[\psi(t,\bm{x})\int \frac{\partial\rho^{\bm{u}}}{\partial t}\differential\bm{x}\right]\!\right)\!\!-\!\!\int_{0}^{T}\!\!\!\!\int_{\mathcal{X}}\!\!\langle\nabla_{\bm{x}}\psi,\bm{f}\rangle\rho^{\bm{u}}\differential\bm{x}\differential t\nonumber\\
&\quad - \int_{0}^{T}\!\!\!\!\int_{\mathcal{X}}\!\langle \mathbf{Hess},\bm{G}\rho^{\bm{u}}\rangle \psi\:\differential \bm{x}\:\differential t\nonumber\\
=& -\!\!\int_{0}^{T}\!\!\!\!\!\int_{\mathcal{X}}\!\!\langle\nabla_{\bm{x}}\psi,\bm{f}\rangle\rho^{\bm{u}}\differential\bm{x}\differential t - \!\!\int_{0}^{T}\!\!\!\!\!\int_{\mathcal{X}}\!\!\langle \mathbf{Hess},\bm{G}\rho^{\bm{u}}\rangle \psi\:\differential \bm{x}\:\differential t
\label{ThirdSummandIBP}
\end{align}
where we assumed that the limits at $\|\bm{x}\|_2\rightarrow r_0$ are zero.

Now consider the second summand in \eqref{ThirdSummandIBP}, and perform two-fold integration by parts w.r.t. $\bm{x}$ as
\begin{align}
&\int_{\mathcal{X}}\!\langle \mathbf{Hess},\bm{G}\rho^{\bm{u}}\rangle \psi\:\differential \bm{x}\nonumber\\
=&\int_{\mathcal{X}}\!\displaystyle\sum_{i,j}\dfrac{\partial^2}{\partial x_i\partial x_j}\left(G_{ij}\rho^{\bm{u}}\right)\psi\:\differential \bm{x}\nonumber\\
=&\displaystyle\sum_{i,j}\!\int_{\mathcal{X}}\!\dfrac{\partial^2}{\partial x_i\partial x_j}\left(G_{ij}\rho^{\bm{u}}\right)\psi\:\differential \bm{x}\nonumber\\ 
=&-\sum_{i,j} \int_{\mathcal{X}} \frac{\partial(G_{ij} \rho^{\bm{u}})}{\partial x_j} \frac{\partial \psi}{\partial x_i} \mathrm{~d} \boldsymbol{x} \nonumber\\
=&\sum_{i,j} \int_{\mathcal{X}}(G_{ij} \rho^{\bm{u}}) \frac{\partial^2 \psi}{\partial x_j \partial x_i} \mathrm{~d} \boldsymbol{x}\nonumber\\
=&\int_{\mathcal{X}} \sum_{i,j}(G_{ij} \rho^{\bm{u}}) \frac{\partial^2 \psi}{\partial x_j \partial x_i} \mathrm{~d} \boldsymbol{x} \nonumber\\
=&\int_{\mathcal{X}}\langle\boldsymbol{G}, \mathbf{Hess}(\psi)\rangle \rho^{\bm{u}} \differential\boldsymbol{x},
\end{align}
which helps rewrite \eqref{ThirdSummandIBP} as
\begin{align}
-\!\!\int_{0}^{T}\!\!\!\!\!\int_{\mathcal{X}}\!\!\langle\nabla_{\bm{x}}\psi,\bm{f}\rangle\rho^{\bm{u}}\differential\bm{x}\differential t - \!\!\int_{0}^{T}\!\!\!\!\!\int_{\mathcal{X}}\!\!\langle\boldsymbol{G}, \mathbf{Hess}(\psi)\rangle \rho^{\bm{u}} \differential\boldsymbol{x}\:\differential t.
\label{ThirdSummandIBPSimplified}   
\end{align}
Combining \eqref{SecondSummandIBP}, \eqref{ThirdSummandIBP}, \eqref{ThirdSummandIBPSimplified}, and dropping the constant term, the}} Lagrangian \eqref{ThreeTermLagrangian} simplifies to
\begin{align}       \int_{0}^{T}\!\!\!\!\int_{\mathcal{X}}\!\!\left(\!\frac{1}{2}\|\bm{u}\|_{2}^{2}-\frac{\partial \psi}{\partial t}-\langle\nabla_{\bm{x}} \psi, \bm{f}\rangle-\langle \bm{G}, \mathbf{Hess}(\psi)\rangle\!\!\right)\! \rho^{\bm{u}} \differential \bm{x}\:\differential t.
        \label{Integrationbyparts}
\end{align}
Minimizing \eqref{Integrationbyparts} w.r.t. $\bm{u}$ for a fixed PDF $\rho^{\bm{u}}$ yields \eqref{OptimalInput}.

We then substitute \eqref{OptimalInput} back in \eqref{Integrationbyparts}, and equate the resulting expression to zero, to arrive at the dynamic programming equation
\begin{align}
        &\int_{0}^{T}\!\int_{\mathcal{X}}\! \left(\frac{1}{2}\|\nabla_{\bm{u}_{\rm{opt}}}\left(\langle\nabla_{\bm{x}} \psi,  \bm{f}\rangle+\langle \bm{G}, \mathbf{Hess}(\psi)\rangle\right)\|_{2}^{2}\right.\nonumber\\
        & -\frac{\partial \psi}{\partial t}- \langle\nabla_{\bm{x}}\psi, \bm{f}\rangle -\left.\langle \bm{G}, \mathbf{Hess}(\psi)\rangle \right) \rho^{\bm{u}}(t,\bm{x}) \differential\bm{x}\:\differential t=0.
    \label{dynamic_programming_equation}
\end{align}
Since \eqref{dynamic_programming_equation} should be satisfied for arbitrary $\rho^{\bm{u}}$, we get
\begin{align*}
\frac{\partial \psi}{\partial t}=&\frac{1}{2} \|\nabla_{\bm{u}}\left(\langle\nabla_{\bm{x}} \psi,  \bm{f}\rangle+\langle \bm{G}, \mathbf{Hess}(\psi)\rangle\right)\|_{2}^{2}-\langle\nabla_{\bm{x}} \psi, \bm{f}\rangle\nonumber\\
&-\langle \bm{G}, \mathbf{Hess}(\psi)\rangle
\end{align*}
which is the HJB PDE \eqref{HJB}. The FPK PDE \eqref{FPK} and the boundary conditions \eqref{boundary_conditions} follow from the primal feasibility conditions \eqref{FPKequation} and \eqref{InitialAndTerminalPDF}, respectively.
\end{proof}

{\blue{\begin{table*}[t]
\centering
\renewcommand\arraystretch{2.5}
{\color{black}{\begin{tabular}{|  c  | 
 c  |  c  |  c  |  c  |  c  |} 
 \hline
 ~Special case~ & ~Drift $\bm{f}$~ & ~Diffusion $\bm{g}$~ & ~Form of \eqref{HJB}~ & ~Form of \eqref{FPK}~ & ~Form of \eqref{OptimalInput}~ \\ [1.5ex] 
 \hline\hline
 OMT \cite{benamou2000computational} & $\bm{u}$ & $\bm{0}_{n\times p}$ & $\dfrac{\partial\psi}{\partial t} + \dfrac{1}{2}\|\nabla_{\bm{x}}\psi\|_2^2 = 0$ & $\dfrac{\partial\rho_{\rm{opt}}^{\bm{u}}}{\partial t}=-\nabla_{\bm{x}}\cdot\left(\rho_{\rm{opt}}^{\bm{u}}\nabla_{\bm{x}}\psi\right)$ & $\bm{u}_{\rm{opt}}=\nabla_{\bm{x}}\psi$\\ [1.5ex]
 \hline
SBP \cite{schrodinger1931umkehrung,schrodinger1932theorie} & $\bm{u}$ & $\bm{I}_{n\times n}$ & $\dfrac{\partial\psi}{\partial t} + \dfrac{1}{2}\|\nabla_{\bm{x}}\psi\|_2^2 + \Delta_{\bm{x}}\psi = 0$ & $\dfrac{\partial\rho_{\rm{opt}}^{\bm{u}}}{\partial t}=-\nabla_{\bm{x}}\cdot\left(\rho_{\rm{opt}}^{\bm{u}}\nabla_{\bm{x}}\psi\right) + \Delta_{\bm{x}}\psi$ & $\bm{u}_{\rm{opt}}=\nabla_{\bm{x}}\psi$\\ [1.5ex]
 \hline
 Control-affine & $\widetilde{\bm{f}}(t,\bm{x})$ & $\bm{B}(t)\in\mathbb{R}^{n\times m}$ & $\dfrac{\partial\psi}{\partial t} + \dfrac{1}{2}\|\bm{B}(t)^{\!\top}\nabla_{\bm{x}}\psi\|_2^2+ \langle\nabla_{\bm{x}}\psi,\widetilde{\bm{f}}\rangle$ & $\dfrac{\partial\rho_{\rm{opt}}^{\bm{u}}}{\partial t}= -\nabla_{\boldsymbol{x}}\cdot\left(\rho^{\boldsymbol{u}}_{\mathrm{opt}}\left(\widetilde{\boldsymbol{f}}+\boldsymbol{B}(t)\boldsymbol{B}(t)^{\top} \nabla_{\boldsymbol{x}} \psi\right)\right)$ & $\bm{u}_{\rm{opt}}=\bm{B}(t)^{\!\top}\nabla_{\bm{x}}\psi$\\ [0.5ex]
 GSBP \cite{caluya2021wasserstein} & $+\bm{B}(t)\bm{u}$ & & $+ \langle\bm{B}(t)\bm{B}(t)^{\!\top},\mathbf{Hess}(\psi)\rangle = 0$ & $+\left\langle\boldsymbol{B}(t)\boldsymbol{B}(t)^{\top}, \mathbf{Hess}\left(\rho^{\boldsymbol{u}}_{\mathrm{opt}}\right)\right\rangle$ & \\[1.5ex]
 \hline
 \end{tabular}}}
 \caption{Special cases of the GSBP \eqref{GeneralizedSBP} (equivalently \eqref{Optimization_FPK}) and corresponding reductions of the optimality conditions \eqref{coupled_PDEs}.}
\label{Table:SpecialCases} 
\end{table*}}}

{\blue{
\begin{remark}\label{Remark:ConditionsOfOptimality}
The conditions of optimality \eqref{coupled_PDEs} relate the primal variables $(\rho^{\bm{u}}_{\rm{opt}}(t,\bm{x}),\bm{u}_{\rm{opt}}(t,\bm{x}))$ with the dual variable (i.e., Lagrange multiplier) $\psi(t,\bm{x})$. Specifically, the HJB PDE \eqref{HJB} and the controlled FPK PDE \eqref{FPK} express the dual and the primal feasibility, respectively. The optimal control policy equation \eqref{OptimalInput} expresses the primal-dual relation.
\end{remark}}}

Structurally, the system of coupled PDEs \eqref{coupled_PDEs} for our control non-affine GSBP is quite different from the  
corresponding system for control-affine SBPs \cite[eq. (5.7)-(5.8)]{chen2021stochastic}, \cite[eq. (20)-(21)]{caluya2021wasserstein}, \cite[eq. (4)]{caluya2021reflected}. In the control-affine SBPs, the conditions of optimality involve two coupled PDEs: one being the HJB PDE and another being the controlled FPK PDE, as in \eqref{HJB}-\eqref{FPK}. Once this pair of PDEs are solved for two unknowns $\rho_{\rm{opt}}^{\bm{u}},\psi$ using techniques such as Hopf-Cole transform followed by a contractive fixed point recursion \cite{chen2016entropic,caluya2021wasserstein} or Feynman-Kac path integrals \cite{nodozi2022schrodinger}, the optimal control $\bm{u}_{\rm{opt}}$ is obtained as a scaled (sub)gradient of $\psi$. In other words, for control-affine SBPs, $\bm{u}_{\rm{opt}}$ is an explicit functional of $\psi$. 

In contrast, the system \eqref{coupled_PDEs} for our non-affine GSBP comprises of $m+2$ coupled PDEs in three unknowns: $\rho_{\rm{opt}},\bm{u}_{\rm{opt}},\psi$, where $m$ is the number of control inputs. This is because \eqref{OptimalInput} itself gives $m$ PDEs coupled in $\psi$ and $\bm{u}_{\rm{opt}}$, while the equation pair \eqref{HJB}-\eqref{FPK} are coupled in $\rho_{\rm{opt}}^{\bm{u}},\bm{u}_{\rm{opt}},\psi$. Existing techniques such as Hopf-Cole transform or Feynman-Kac path integrals no longer apply for this situation, and new ideas are needed to numerically solve the coupled system \eqref{coupled_PDEs}-\eqref{boundary_conditions}.

{\blue{Table \ref{Table:SpecialCases} summarizes how known results in the literature can be recovered as special cases of \eqref{coupled_PDEs}.}}

%%%%%%%%%%%%%%%%%%%%%%%

%%%%%%%%%%%%%%%%%%%%%%%%%%%%%%%%%%%

\section{Solving the Conditions for Optimality \\using PINN with Sinkhorn Losses}\label{secPINN}

In this Section, we propose a new variant of the PINN \cite{raissi2019physics,lu2021deepxde} designed for numerically solving \eqref{coupled_PDEs}-\eqref{boundary_conditions}. To do so, we first introduce the 2-Wasserstein distance followed by its entropic regularization.

\begin{definition}\label{defWass}
(\textbf{2-Wasserstein distance}) The \emph{squared 2-Wasserstein distance} $W$ between a pair of probability measures $\mu_1,\mu_2$ supported respectively on $\mathcal{X}, \mathcal{Y}\subseteq \mathbb{R}^n$, is
\begin{align}
W^{2}\left(\mu_1,\mu_2\right) := \underset{\mu\in\mathcal{M}\left(\mu_1,\mu_2\right)}{\inf}\:\displaystyle\int_{\mathcal{X}\times\mathcal{Y}}\|\bm{x}-\bm{y}\|_2^2\:\differential\mu(\bm{x},\bm{y})
\label{DefWassContinuous}    
\end{align}
where $\mathcal{M}\left(\mu_1,\mu_2\right)$ is the set of joint probability measures or couplings over the product space $\mathcal{X}\times\mathcal{Y}$ having $\bm{x}$ marginal $\mu_1$, and $\bm{y}$ marginal $\mu_2$. Hereafter, we refer to \eqref{DefWassContinuous} as the squared Wasserstein distance, dropping the prefix 2.
\end{definition}
For metric properties of $W$, see e.g., \cite[Ch. 7]{villani2003topics}. Whenever $\mu_1$ and $\mu_2$ are absolutely continuous, their respective PDFs $\rho_1,\rho_2$ exist, i.e., $\differential \mu_1(\bm{x})=\rho_1(\bm{x})\differential \bm{x}$ and $\differential \mu_2(\bm{y})=\rho_2(\bm{y})\differential \bm{y}$), and we use the equivalent notation $W^{2}\left(\rho_1,\rho_2\right)$. Notice that \eqref{DefWassContinuous} corresponds to a linear program (LP) and is in fact, the Kantorovich formulation \cite{kantorovich1942transfer} of OMT.

\begin{definition}\label{defSinkhornLoss}
(\textbf{Sinkhorn loss})
The Sinkhorn loss between a pair of probability measures $\mu_1,\mu_2$ supported respectively on $\mathcal{X}, \mathcal{Y}\subseteq \mathbb{R}^n$, with fixed $\varepsilon>0$, is the entropy-regularized squared Wasserstein distance, i.e.,
\begin{align}
W_{\varepsilon}^{2}\left(\mu_1,\mu_2\right) := &\underset{\mu\in\mathcal{M}\left(\mu_1,\mu_2\right)}{\inf}\displaystyle\int_{\mathcal{X}\times\mathcal{Y}}\bigg\{\|\bm{x}-\bm{y}\|_2^2+\varepsilon\log\mu\bigg\}\nonumber\\
&\qquad\qquad\qquad\qquad\qquad\differential\mu(\bm{x},\bm{y}),
\label{DefSinkhornLoss}    
\end{align}
where $\mathcal{M}\left(\mu_1,\mu_2\right)$ is the set of joint probability measures or couplings over the product space $\mathcal{X}\times\mathcal{Y}$ having $\bm{x}$ marginal $\mu_1$, and $\bm{y}$ marginal $\mu_2$.
\end{definition}
It is known \cite{carlier2017convergence} that $W_{\varepsilon}^{2}\rightarrow W^2$ in the limit $\varepsilon\downarrow 0$. Even though the Sinkhorn loss \eqref{DefSinkhornLoss} does not define a metric over $\mathcal{M}$, its computation offers several advantages over that of \eqref{DefWassContinuous}. For instance, the entropic regularization makes the objective in \eqref{DefSinkhornLoss} strictly convex, and its discrete implementation was proposed \cite{cuturi2013sinkhorn} as a fast numerical approximant of the OMT \eqref{DefWassContinuous}. As we explain next, \eqref{DefSinkhornLoss} is also better suited for automatic differentiation w.r.t. neural network parameters for PINN training, which is what we need for our boundary conditions \eqref{genSBPconstr} (or equivalently \eqref{boundary_conditions}).

\begin{figure*}[t]
\centering
\includegraphics[width=.8\paperwidth]{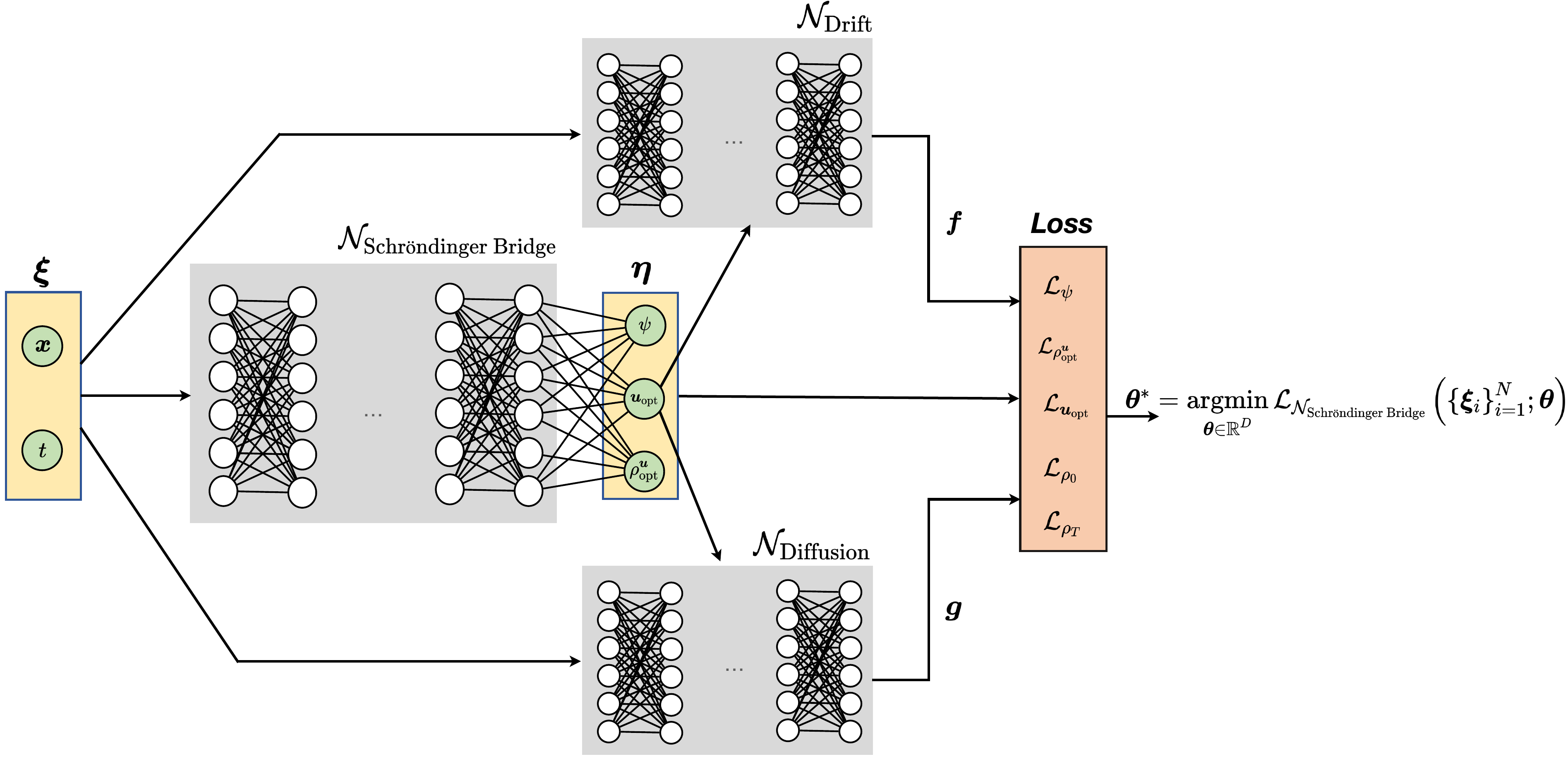}
\caption{{\small{The architecture of the physics-informed neural network with the system state $\bm{x}$, and the time $t$ as the input features $\bm{\xi}:=(\bm{x}, t)$. The network output $\bm{\eta}$ comprises of the value function, optimally controlled PDF, and optimal control policy, i.e., $\bm{\eta}:=(\psi,\rho^{\bm{u}}_{\rm{opt}}, \bm{u}_{\rm{opt}})$. The networks $\mathcal{N}_{\text {Drift}}$ and $\mathcal{N}_{\text {Diffusion}}$ are fully trained from MD simulation.}}}
\label{fig:PINNStructure}
\end{figure*}

\subsection{Learning with Sinkhorn Losses}
\label{LearningWithSinkhornLoss}
To better understand the advantage of learning with Sinkhorn losses, consider the squared {\blue{Euclidean}} distance matrix $\bm{C}\in \mathbb{R}^{ d \times d}$,
and for a given pair of $d$-dimensional probability vectors $\bm{\mu}_1,\bm{\mu}_2$, let $\Pi (\bm{\mu}_1,\bm{\mu}_2)$ denote the set of all coupling matrices, i.e.,
\begin{align}
  \Pi (\bm{\mu}_1,\bm{\mu}_2):=\bigg\{\bm{M} \in &\mathbb{R}^{ d \times d}\mid \bm{M}\geq 0\:\text{(element-wise)},\nonumber\\
  &\bm{M} \bm{1}=\bm{\mu}_1, \bm{M}^{\top} \bm{1}=\bm{\mu}_2\bigg\}.
\end{align}
The dimension $d$ here represents the number of samples involved, i.e., the dimensionality of the standard simplex in which $\bm{\mu}_1,\bm{\mu}_2$ belong to. Then the discrete version of \eqref{DefSinkhornLoss} becomes
\begin{align}     W^{2}_{\varepsilon}\left(\bm{\mu}_1,\bm{\mu}_2\right)=&\min_{\bm{M} \in \Pi (\bm{\mu}_1,\bm{\mu}_2)} \langle \bm{C} +\varepsilon \log \bm{M},\bm{M}\rangle
     \label{WassDiscEntropy}
\end{align}
where $\varepsilon>0$ is a fixed regularization parameter. The convex problem \eqref{WassDiscEntropy} can be solved using the Sinkhorn recursions \cite{sinkhorn1964relationship,sinkhorn1967diagonal} a.k.a. iterative proportional fitting procedure (IPFP) \cite{deming1940least}.
These recursions are motivated by the observation that the minimizer of \eqref{WassDiscEntropy} must be a diagonal scaling of the known matrix $\bm{\Gamma}:=\exp\left(\frac{-\bm{C}}{2\varepsilon}\right)\in\mathbb{R}^{d\times d}_{>0}$ where the exponential is element-wise, i.e.,
\begin{align}
    \bm{M}= {\rm{diag}}(\bm{v}_1)\:\bm{\Gamma}\:{\rm{diag}}(\bm{v}_2)
\end{align}
for to-be-determined $\bm{v}_1,\bm{v}_2\in\mathbb{R}^{d}$.

Starting with some initial guess, the Sinkhorn recursions alternate between updating $\bm{v}_{1}$ and $\bm{v}_{2}$ until convergence:
\begin{subequations}
\begin{align}
    &\bm{v}_{1}^{k+1} \leftarrow \bm{\mu}_1 \oslash \left( \bm{\Gamma} \bm{v}_{2}^{k}\right),\label{uUpdate}\\
    &\bm{v}_{2}^{k+1} \leftarrow \bm{\mu}_2 \oslash \left( \bm{\Gamma}^{\top} \bm{v}_{1}^{k+1}\right),\label{vUpdate}
\end{align}
\label{SinkRecursions}
\end{subequations}
for recursion index $k=0,1,\hdots$; the symbol $\oslash$ denotes the element-wise (Hadamard) division. The updates \eqref{uUpdate}-\eqref{vUpdate} can be seen as alternating Kullback-Leibler projections \cite{bregman1967relaxation,benamou2015iterative} with guaranteed linear convergence.

When $\varepsilon=0$ in \eqref{WassDiscEntropy}, we get an LP corresponding to the discrete version of \eqref{DefWassContinuous}. This LP has $d^2$ unknowns with $d^2 + 2d$ constraints, and solving the same as standard network flow problem has $\widetilde{O}\left(d^{2}\sqrt{2d}\right)$ complexity \cite{lee2014path} which is impractical for large $d$. Furthermore, using \eqref{DefWassContinuous} as the endpoint loss for training a PINN to learn the solution of \eqref{coupled_PDEs}-\eqref{boundary_conditions}, requires us to compute 
\begin{align}
{\texttt{AutoDiff}}_{\bm{\theta}}W^{2}\left(\mu_{i},\mu_{i}^{\text{epoch index}}(\bm{\theta})\right) \:\forall\:i\in\{0,T\}
\label{AutodiffWassSquared}
\end{align}
for each epoch of the training, where ${\texttt{AutoDiff}}_{\bm{\theta}}$ refers to the standard reverse mode automatic differentiation w.r.t. PINN training parameter $\bm{\theta}$. Evaluating \eqref{AutodiffWassSquared} then amounts to differentiating through a very large scale LP which is computationally challenging even for moderately large $d$.

In contrast, using \eqref{DefSinkhornLoss} as the endpoint loss for training a PINN to learn the solution of \eqref{coupled_PDEs}-\eqref{boundary_conditions}, requires us to compute 
\begin{align}
{\texttt{AutoDiff}}_{\bm{\theta}}W_{\varepsilon}^{2}\left(\mu_{i},\mu_{i}^{\text{epoch index}}(\bm{\theta})\right) \:\forall\:i\in\{0,T\}
\label{AutodiffSinkhorn}
\end{align}
for a fixed $\varepsilon>0$. Because the Sinkhorn recursions \eqref{SinkRecursions} involve a series of differentiable linear operations, it is amenable to Pytorch auto-differentiation to support backpropagation. Thus using $W_{\varepsilon}^2$ instead of $W^2$ as the endpoint distributional losses incur lesser computational overhead allowing us to train PINNs for nontrivial GSBPs. This advantage of Sinkhorn losses over Wasserstein losses, has also been pointed out in a different context in \cite{genevay2018learning}. Rigorous consistency results have appeared in \cite{pauwels2022derivatives} showing that the derivatives of the iterates from Sinkhorn recursion computed through automatic differentiation, indeed converge to the derivatives of the corresponding Sinkhorn loss.

\subsection{Proposed PINN Architecture}
\label{subsecProposedPINN}
Our proposed architecture for the PINN is shown in  Fig. \ref{fig:PINNStructure}. For {\blue{the}} GSBP {\blue{considered here}}, the state-time $\bm{\xi}:=(\bm{x}, t)$ comprises the features that are inputs to the network, and the network output $\bm{\eta}:=(\psi,\rho^{\bm{u}}_{\rm{opt}}, \bm{u}_{\rm{opt}})$ comprises of the value function, optimally controlled PDF, and optimal policy. 

The proposed PINN is a fully connected feed-forward NN with multiple hidden layers, and we parameterize its output using the network parameter $\bm{\theta} \in \mathbb{R}^{D}$, i.e.,
\begin{align}
\bm{\eta}(\bm{\xi}) \approx \mathcal{N}_{\text{Schr\"{o}dinger Bridge}}(\bm{\xi} ; \bm{\theta}),
\end{align}
where $\mathcal{N}_{\text{Schr\"{o}dinger Bridge}}(\cdot;\bm{\theta})$ denotes the NN approximant parameterized by $\bm{\theta}$. Here $D$ denotes the dimension of the parameter space, i.e., the total number of to-be-trained weight, bias and scaling parameters for the NN. For all neurons, we use the $\tanh$ activation functions.

As mentioned in Sec. \ref{SecNeuralSBP}, the explicit expressions for $\bm{f}$ and $\bm{g}$, the drift and diffusion coefficients, are not available from first-principle physics. We learn these coefficients from MD simulation data (see Sec. \ref{subsecSystemSetup}, \ref{subsecLearningDyn}). As shown in Fig. \ref{fig:PINNStructure}, the networks $\mathcal{N}_{\text {Drift}}$ and $\mathcal{N}_{\text {Diffusion}}$, represent the learnt drift $\bm{f}$ and the learnt diffusion $\bm{g}$, respectively, which are used to evaluate the loss function $\mathcal{L}_{\mathcal{N}_{\text{Schr\"{o}dinger Bridge}}}$ for the PINN.

The PINN loss function $\mathcal{L}_{\mathcal{N}_{\text{Schr\"{o}dinger Bridge}}}$ consists of the sum of the losses associated with the $m+2$ equations in \eqref{coupled_PDEs}, and the losses associated with the boundary conditions \eqref{boundary_conditions}. Specifically, let $\mathcal{L}_{\psi}$ be the MSE loss for the HJB PDE \eqref{HJB}. Likewise, let $\mathcal{L}_{\rho^{\bm{u}}_{\rm{opt}}}$ be the MSE loss for the FPK PDE \eqref{FPK}, and because the control policy has $m$ components $(u_1,\ldots,u_m)$, let $\mathcal{L}_{u_{j_{\rm{opt}}}}\left.\right|_{j=1,\ldots,m}$ be the corresponding MSE loss term for each control policy component in \eqref{OptimalInput}. 

However, the MSE losses are insufficient to capture the distributional mismatch for endpoint boundary conditions \eqref{boundary_conditions}. Per Sec. \ref{LearningWithSinkhornLoss}, we use 
the Sinkhorn losses as the boundary condition losses $\mathcal{L}_{\rho_{0}}$ and  $\mathcal{L}_{\rho_{T}}$, and differentiate through the corresponding Sinkhorn recursions. 

Thus,
\begin{align}
\mathcal{L}_{\mathcal{N}_{\text{Schr\"{o}dinger Bridge}}}:=\mathcal{L}_{\rho_{0}}+\mathcal{L}_{\rho_{T}}+\mathcal{L}_{\psi}+\mathcal{L}_{\rho^{\bm{u}}_{\rm{opt}}}+\displaystyle\sum_{j=1}^{m}\mathcal{L}_{u_{j_{\rm{opt}}}},
\label{totalloss}
\end{align} 
where each summand loss term in \eqref{totalloss} is evaluated on a set of $N$ collocation points $\{\bm{\xi}_i\}_{i=1}^{N}$ in the domain of the feature space $\Omega:=\mathcal{X}\times[0,T]$, i.e., $\{\bm{\xi}_i\}_{i=1}^{N}\subset\Omega$. For instance, the equation error losses are of the form
\begin{align*}
&\mathcal{L}_{\psi} :=\frac{1}{N} \sum_{i=1}^{N}\left(\left.\frac{\partial \psi}{\partial t}\right|_{\bm{\xi}_{i}}-\left.\frac{1}{2} \|\bm{u}_{\rm{opt}}\|_{2}^{2}\right|_{\bm{\xi}_{i}}\left.+\langle\nabla \psi, \bm{f}\rangle \right|_{\bm{\xi}_{i}}\right.\\
&\qquad \qquad\qquad\qquad\qquad\qquad\left.\left.+\langle \bm{G}, \mathbf{Hess}(\psi)\rangle \right|_{\bm{\xi}_{i}}\right)^{2}, 
\end{align*}
\begin{align*}
&\mathcal{L}_{\rho^{\bm{u}}_{\rm{opt}}} :=\frac{1}{N} \sum_{i=1}^{N}\left(\left.	\frac{\partial \rho^{\bm{u}}_{\rm{opt}}}{\partial t}\right|_{\bm{\xi}_{i}}+\left.\nabla .( \rho^{\bm{u}}_{\rm{opt}} \bm{f})\right|_{\bm{\xi}_{i}}\right.\\
&\qquad \qquad\qquad\qquad\qquad\left.\left.-\langle \mathbf{H e s s},\bm{G}\rho^{\bm{u}}_{\rm{opt}}\rangle\right|_{\bm{\xi}_{i}}\right)^{2},
\end{align*}
\begin{align*}
&\mathcal{L}_{u_{j_{\rm{opt}}}}\left.\right|_{j=1,\ldots,m} :=\frac{1}{N} \sum_{i=1}^{N}\left(	u_{j_{\rm{opt}}}\right|_{\bm{\xi}_{i}}-\left.\frac{\partial}{\partial u_{j_{\rm{opt}}}}\left(\langle\nabla_{\bm{x}} \psi,  \bm{f}\rangle \right. \right.\\
&\qquad \qquad\qquad\qquad\qquad\left.\left.\left.+\langle \bm{G}, \mathbf{Hess}(\psi)\rangle\right)\right|_{\bm{\xi}_{i}}\right)^{2},
\end{align*}
where $u^{j}_{\rm{opt}}$ denotes the $j$th component of the optimal control $\bm{u}_{\rm{opt}}$.
% \begin{align*}
% \mathcal{L}_{\rho_{0}}:=W^{2}_{\varepsilon}\left(\rho^{\bm{u}}_{\rm{opt}}(\bm{x},t=0),\rho_0\right),
% \end{align*}
% \begin{align*}
% \mathcal{L}_{\rho_{T}}:=W^{2}_{\varepsilon}\left(\rho^{\bm{u}}_{\rm{opt}}(\bm{x},t=T),\rho_T\right)
% \end{align*}
%for each collocation point $\bm{\xi}_i$, $i=1,\hdots,n$.

% Since PINN's activations function is $\tanh(\cdot)$, its output tensors could be negative, positive, trivial, take on any distribution shape, and may not be a valid PDF during training. Therefore, computing a distance measure between such an output can result in numerical singularities or non-differentiable distances or local minima distances that jeopardize overall PINN training. 

We implemented the Sinkhorn recursions with the log-sum-exp (LSE) technique \cite[Section 3]{viehmann2019implementation} to maintain numerical stability %during iteration irrespective of the input numerical properties 
at the expense of minor memory overhead. 
%We expose $\rho_0$ and $\rho_T$ to the network during training, and for every epoch, we compute the Sinkhorn regularized squared Wasserstein distance respective network output to these two distributions.  
We {\blue{employed}} mini-batching for sampling our PINN output, and used the same sample indices to sample from our prescribed $\rho_{0},\rho_{T}$. {\blue{The}} squared Euclidean distance matrix $\bm{C}$ mentioned in Sec. \ref{LearningWithSinkhornLoss} {\blue{was constructed}} from the output batch points.

% While the LSE technique is a good way to guarantee numerical stability for arbitrary neural network output distributions, it does not guarantee that the computed distance will not be a local minimum if the input distributions are not valid PDFs. It is still necessary to introduce another term to the loss function to penalize invalid distribution outputs, specifically negative outputs. We express this loss component as the negative sum of all the outputs that violate this constraint. We weight this loss component by a scalar multiplier (of 10) so that the network penalizes violating this constraint more and thus learns it 'before' learning the Sinkhorn regularized squared Wasserstein distance. We also constrain the boundary condition distributions to have a trapze of 1 to satisfy the PDF definition. 

We used the PINN software library \cite{lu2021deepxde} with a Pytorch backend to perform numerical experiments using the above loss functions. The PINN library \cite{lu2021deepxde} was not written for {\blue{Schr\"{o}dinger bridge}}-type problems, so we needed to modify it to suit our needs. One modification was to program PINN to compute loss between outputs and distributions directly and integrate the Sinkhorn iteration algorithm into the library. We also modified it to perform the mini-batching we needed.
In summary, for training the PINN, the overall loss \eqref{totalloss} {\blue{was minimized}} over $\bm{\theta}\in\mathbb{R}^{D}$ by solving
\begin{align}
\bm{\theta}^{*}=\underset{\bm{\theta} \in \mathbb{R}^{D}}{\operatorname{argmin}}\: \mathcal{L}_{\mathcal{N}_{\text{Schr\"{o}dinger Bridge}}}(\{\bm{\xi}_i\}_{i=1}^{N}; \bm{\theta}).    
\label{NNtraining}    
\end{align}
{\blue{The next section details the simulation setup and reports the numerical results.}}
%%%%%%%%%%%%%%%%%%%%%%%%%%%%%%%%%%%%%%%%%%%%%%%%%%%%%%%%%%%%%%%%%%%%%%%%%%%%%%%%

\section{Numerical Case Study of Controlled Isotropic Colloidal SA in an NPT ensemble}\label{SecNumerical}

{\blue{We now present a numerical case study of a colloidal SA system where the drift coefficient $\bm{f}$ and the diffusion coefficient $\bm{g}$ are not analytically available, instead they are learnt as NN representations $\mathcal{N}_{\text {Drift}}$ and $\mathcal{N}_{\text {Diffusion}}$, from MD simulation data. Such learnt representations are nonlinear in the state $\bm{x}$ and non-affine in control $\bm{u}$. We then solve the GSBP \eqref{GeneralizedSBP} using the PINN architecture proposed in Sec. \ref{secPINN} to design a minimum effort controller steering the distribution in the order-parameter space to synthesize the body-centered cubic (BCC) crystal structure over the given time horizon. Fig. \ref{fig:BCC} shows an initial disordered structure and a final BCC structure.}}
\subsection{System Description}\label{subsecSystemSetup}
We consider the \emph{in-silico} representation of isotropic colloidal particles with identical Lennard-Jones interaction potentials within an NPT (isothermal-isobaric) ensemble. The Lennard-Jones potential is used to model particle interactions in the system and is defined as
\begin{equation} U(r) := 4\epsilon \left(\left(\frac{\sigma}{r}\right) ^ {12} - \left(\frac{\sigma}{r}\right) ^ 6\right),
\end{equation}
where $r$ denotes the particle radius, and $\epsilon$ denotes the depth of the potential energy well and thus quantifies the strength of attractive forces between particles. The symbol $\sigma$ denotes the distance at which the potential energy is nullified, thereby demarcating the intermolecular potential's shift from attraction to repulsion depending on particle size \cite[p. 234]{mcquarrie2000statistical}.

\begin{figure}[t]
\centering
\begin{subfigure}{.22\textwidth}
  \centering
  \includegraphics[width=\linewidth]{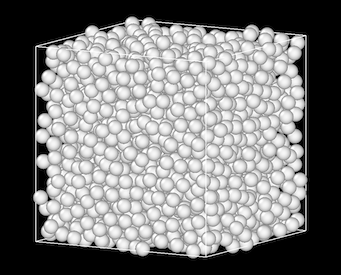}
  \caption{{\small{}}}
  \label{fig:Disorder}
\end{subfigure}%
\begin{subfigure}{.22\textwidth}
  \centering
  \includegraphics[width=\linewidth]{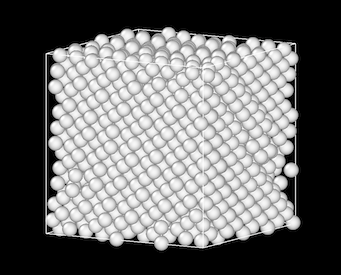}
  \caption{{\small{}}}
  \label{fig:BCC}
\end{subfigure}
\caption{{\small{{\blue{(a) An initial disordered crystalline structure. (b) A final BCC structure with minor defects. These images were generated using OVITO \cite{Ovito}.}}}}}
\label{fig:BCC}
\end{figure}
 
An ensemble of 2048 particles is initialized at a given temperature and pressure. While the positions of these particles may be considered as the most natural states of a colloidal SA system, they result in an unmanageably high-dimensional state space. To circumvent this difficulty, we seek a lower-dimensional representation. To this end, the Steinhardt bond order parameters $\langle C_{10}\rangle$ and $\langle C_{12}\rangle$ are used in this work, which are directly defined in terms of the particle positions. To calculate these parameters \cite{Steinhart1983} from the MD simulation data, we proceed through a series of steps, as discussed below. 

We first extract the positional information for each particle from the MD simulation data (see Sec. \ref{subsecLearningDyn}). Next, we identify the neighbors for each particle based on the Voronoi method \cite{rycroft2009voro++}. Using this information, we calculate the spherical harmonics, $Y_{lm}$, indexed by two quantum numbers, viz. the azimuthal or orbital quantum number, denoted by $l$, and the magnetic quantum number, denoted by $m$. 

The azimuthal quantum number defines the shape of the orbital, and for our context $l \in [1, 12]$. The magnetic quantum number represents the orientation of the orbital in space, and for our context $m \in [-l, l]$, see e.g., \cite[p. 545]{McquarriePchem}. Accordingly, the spherical harmonics are defined as
\begin{align} 
Y_{lm}(\theta, \phi) := \sqrt{\frac{{2l + 1}}{{4\pi}}  \frac{{(l - m)!}}{{(l + m)!}}}  P^m_l(\cos \theta)  e^{im\phi},
\label{defYlm}
\end{align}
where $\theta$ and $\phi$ represent the polar and azimuthal angles, respectively. In \eqref{defYlm}, the $P^m_l$ denote the associated Legendre polynomials \cite[p. 331--339]{abramowitz1968handbook}, which is a class of functions that arise in the solution to Laplace's equation in spherical coordinates.

Let $\nu(i)$ denote the number of neighbors of particle $i$, and let $r_{ij}$ signify the positional vector between particle $i$ and its neighbor $j$. Subsequently, the $l$th bond order parameter $C_{l}(i)$, for each particle $i$, is computed as \cite{Steinhart1983}
\begin{align}
 C_{l}(i) =\left(\frac{4 \pi}{2 l+1} \displaystyle\sum_{m=-l}^l\left|\frac{1}{\nu(i)} \sum_{j=1}^{\nu(i)} Y_{l m}\left(r_{ij}\right)\right|^2\right)^{\frac{1}{2}},
 \label{Cli}
\end{align}
where the index $i \in [0, \nu]$ and $\nu$ represents the total number of 
particles in the ensemble ($\nu=2048$ in our case study). Furthermore, index $j \in [0, \nu(i)]$. In \eqref{Cli}, normalizing by the number of neighbors ensures that the final system order parameter is size-independent and thus, scalable across different systems. That is, the normalization ensures that $C_{l}(i) \in [0,1]$. 

Next, the individual bond order parameters $C_{l}(i)$ are averaged over all particles in the ensemble to calculate the averaged $l$th Steinhardt bond order parameter 
\begin{equation}
\langle C_l \rangle=\frac{1}{\nu} \sum_{i=1}^{\nu} C_{l}(i),
\end{equation}
which can be used to describe the state of a colloidal SA system.\footnote{{\blue{In actual self-assembly systems, image analysis techniques can be used to locate and track particle ce
nters, as well as compute local and global order parameters in real-time (e.g., see \cite{juarez2012feedback}).}}} Therefore, the physics-based order parameters serve as a reduced-dimensionality conduit that enhances the efficiency and effectiveness of our subsequent analyses by circumventing the need to work with high-dimensional particle position data.\footnote{In this work, the Steinhart bond order parameters were calculated using the Python package Freud \cite{freud2020}.}

In this work, we specifically choose the order parameters $\langle C_{10}\rangle$ and $\langle C_{12}\rangle$ for their efficacy in distinguishing between the body-centered cubic (BCC) and the face-centered cubic (FCC) structures. The values of $\langle C_{10}\rangle$ and $\langle C_{12}\rangle$ for defect-free assembled BCC and FCC structures do not overlap, which enables differentiation between the two structure types.

In summary, the controlled dynamics of our colloidal SA system is described by the SDE~\eqref{GenSBPdyn}, where the state and inputs are defined as  
\begin{align*}
\bm{x}&:=(\langle C_{10}\rangle,\langle C_{12}\rangle)\in \mathcal{X}\equiv [0,1]^2,\\
\bm{u}&:=(\text{temperature}, \text{pressure})\in\mathcal{U}.
\end{align*}
We denote the components of the optimal control policy $\bm{u}_{\rm{opt}}$ as $u_{1}^{\rm{opt}}, u_{2}^{\rm{opt}}$ respectively.

\subsection{Learning $\bm{f}$ and $\bm{g}$}\label{subsecLearningDyn}

\begin{figure}[t]
\centering
\includegraphics[width=.4\paperwidth]{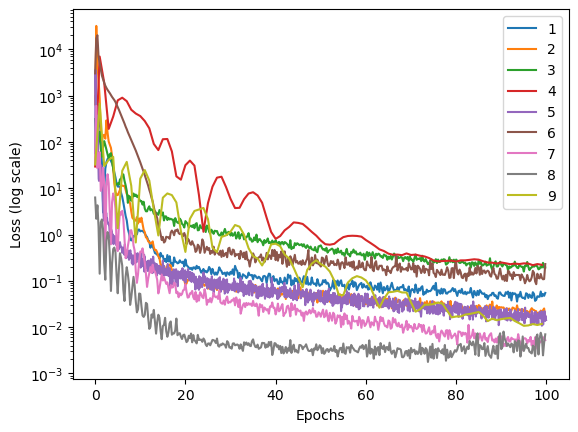}
\caption{{\small{Validation losses for nine different neural network models $\mathcal{N}_{\text {Drift}}$ and $\mathcal{N}_{\text {Diffusion}}$, for the drift and diffusion terms in the SDE~\eqref{GenSBPdyn}, with legend numbers corresponding to model numbers in Table \ref{tab:model_comparison}.}}}
\label{fig:validation_loss}
\end{figure}

To learn the NN representations $\mathcal{N}_{\text {Drift}}$ and $\mathcal{N}_{\text {Diffusion}}$, which model the drift and diffusion coefficients  $\boldsymbol{f}$ and $\boldsymbol{g}$ in the SDE~\eqref{GenSBPdyn}, we performed MD simulations for the above-described system using the Python package HOOMD-blue \cite{ANDERSON2020BLUE} with final time $T=200$ s. The simulation data consisted of 200 state trajectories, i.e., the trajectories of the order parameters $\langle C_{10}\rangle,\langle C_{12}\rangle$ for $t\in[0,T]$, that represent the time evolution of position of the $\nu=2048$ particles of the colloidal SA system, mentioned earlier in Sec. \ref{subsecSystemSetup}. Each state trajectory was generated under different linear temperature and pressure ramp rates (i.e., $\bm{u}$), which were sampled using a Latin Hypercube design and scaled to $[-0.005, 0.005]$, the input range for the simulation.  To generate the training and test data for learning the NN models, the state trajectories were sampled 500 times. %A sensitivity analysis informed the selection of 500 as the optimal sampling interval. 

Building on our earlier work \cite{O_Leary_2022}, $\mathcal{N}_{\text {Drift}}$ and $\mathcal{N}_{\text {Diffusion}}$ were trained on the MD data with a controlled neural SDE. The NNs are designed to be functions of the current time $t\in[0,T]$, the system state $\bm{x}$, and the control input $\bm{u}$. The networks are passed to \eqref{GenSBPdyn} to predict the state evolution. %Thus, the structure of the SDE is enforced in the training procedure. 
The MSE loss is computed for each time step, and the learning process aims to minimize the total MSE loss between the networks' predicted states and the actual observed states from the MD simulation trajectories. For model optimization, we used the Adam optimizer \cite{kingma2014adam} which adjusts the learning rate on a per-parameter basis. The learning rate was initially set to a predefined constant, as per Table \ref{tab:model_comparison}, and was subsequently adjusted using an exponential learning rate scheduler with a decay rate of 0.999. This scheduler reduces the learning rate multiplicatively after each epoch. The data was partitioned into a 70/20/10 distribution for the training, testing, and validation subsets, respectively. The implementation of these networks was done with the torchsde \cite{li2020scalable} Python package.

To determine the best architecture for the NNs $\mathcal{N}_{\text {Drift}}$ and $\mathcal{N}_{\text {Diffusion}}$, we used hyperparameter (depth and width, batch size, learning rate) turning as detailed in Table \ref{tab:model_comparison}; a total of nine models were trained and evaluated. All of the NN architectures follow a sequential design of fully connected layers, with 5 input units and an output layer of 2 units. The architectures vary in the number of hidden layers and their nodes, all using $\tanh$ activation functions. Architecture 1 employs one hidden layer with 200 nodes; architecture 2 utilizes a hidden layer of 1000 nodes; and architecture 3 deploys six hidden layers with 200 nodes each. The batch size, defining the number of samples to be processed before updating the model, is tuned for learning. Lastly, we adjust the learning rate, a factor determining how much the model's parameters should be adjusted with respect to the calculated error, for balanced and steady learning without risking instability or slow convergence. The MSE was used as the loss function for all models.

Fig. \ref{fig:validation_loss} shows the validation MSE loss for all models, which are evaluated by using the $\mathcal{N}_{\text {Drift}}$ and $\mathcal{N}_{\text {Diffusion}}$ in \eqref{GenSBPdyn} to predict the state $\bm{x}$, and then comparing the predicted states with those obtained from MD simulations. These validation results demonstrate that all models converge, indicating that the training time was sufficient.  %Thus, Figure \ref{fig:validation_loss} represents the validation of both $\mathcal{N}_{\text {Drift}}$ and $\mathcal{N}_{\text {Diffusion}}$. 
%After training each NN representation, the resulting loss values are evaluated, with the representation that achieves the smallest validation loss chosen as the optimal one.
As seen in Table \ref{tab:model_comparison} and Fig. \ref{fig:validation_loss}, model 7 exhibited the best performance evidenced by its minimal validation loss. Consequently, we used model 7 for representing the colloidal SA dynamics in the form of \eqref{GenSBPdyn}. Its corresponding $\mathcal{N}_{\text {Drift}}$ and $\mathcal{N}_{\text {Diffusion}}$ are used for the optimal control synthesis for the GSBP. {\blue{On an NVIDIA GTX 1080, each model undergoes training that, on average, takes 10 seconds per training step. To complete 100 epochs, this process requires approximately 1.2 hours per model. The approximate inference time for the model is 0.0123 seconds.}}

\begin{table}[t!]
\centering
\begin{tabular}{| c | c | c | c | c | c |}
\hline
\begin{tabular}[c]{@{}l@{}}Model \\ number\end{tabular} & \begin{tabular}[c]{@{}c@{}}Learning \\rate\end{tabular} & \begin{tabular}[c]{@{}c@{}}Batch size\\  (of epoch)\end{tabular} & \begin{tabular}[c]{@{}c@{}}Model \\ architecture\end{tabular} & \begin{tabular}[c]{@{}c@{}}Validation \\ loss\end{tabular} & \begin{tabular}[c]{@{}c@{}}Training \\ loss\end{tabular} \\ \hline
$1$                                                     & $10^{-3}$                                                     & $1/4$                                                              & $1$                                                             & $0.390$                                                    & $0.057$                                                     \\ \hline
$2$                                                       & $10^{-2}$                                                     & $1/4$                                                              & $1$                                                             & $0.460$                                                      & $0.024$                                                     \\ \hline
$3$                                                       & $10^{-4}$                                                     & $1/4$                                                              & $1$                                                             & $1.260$                                                       & $0.120$                                                      \\ \hline
$4$                                                       & $10^{-3}$                                                     & $1$                                                                & $1$                                                             & $3.890$                                                       & $0.200$                                                   \\ \hline
$5$                                                       & $10^{-3}$                                                     & $1/4$                                                             & $1$                                                             & $0.200$                                                       & $0.016$                                                     \\ \hline
$6$                                                       & $10^{-3}$                                                     & $1/4$                                                              & $2$                                                             & $9.290$                                                       & $0.072$                                                   \\ \hline
$7$                                                       & $10^{-3}$                                                     & $1/4$                                                              & $3$                                                             & $0.030$                                                       & $0.003$                                                    \\ \hline
$8$                                                       & $10^{-4}$                                                     & $1/4$                                                              & $3$                                                             & $0.031$                                                      & $0.007$                                                     \\ \hline
$9$                                                       & $10^{-3}$                                                     & $1$                                                                & $3$                                                             & $0.110$                                                      & $0.015$                                                     \\ \hline
\end{tabular}
\caption{Comparison of different model architectures and hyperparameters for learning the NN representations $\mathcal{N}_{\text {Drift}}$ and $\mathcal{N}_{\text {Diffusion}}$ for the drift $\bm{f}$ and diffusion $\bm{g}$, respectively. The different architectures vary in the number of hidden layers and their nodes, all using $\tanh$ activation function. Architecture 1 employs one hidden layer with 200 nodes, architecture 2 utilizes a hidden layer of 1000 nodes, and architecture 3 deploys six hidden layers with 200 nodes each.}
\label{tab:model_comparison}
\end{table}

%Comparison of different hyperparameters for learning rate, batch size, and model architecture. Batch size is reported as a fraction of epoch size. All model architectures follow a sequential design of fully connected layers, with 5 input units and an output layer of 2 units. The architectures vary in the number of hidden layers and their nodes, all using a hyperbolic tangent activation function. Architecture 1 employs one hidden layer with 200 nodes, architecture 2 utilizes a hidden layer of 1000 nodes, while architecture 3 deploys six hidden layers with 200 nodes each.

\begin{figure*}[htbp!]
\centering
\includegraphics[width=.9\linewidth]{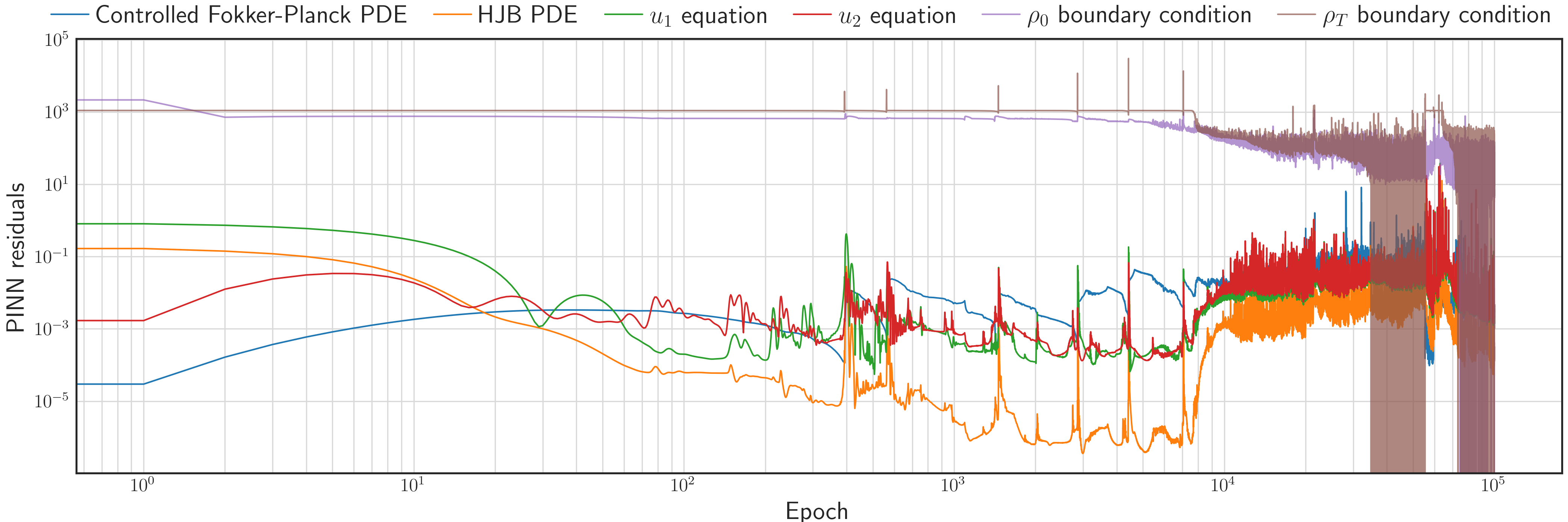}
\caption{{\small{The PINN residuals in solving the conditions of optimality \eqref{coupled_PDEs}-\eqref{boundary_conditions} for the simulation in Sec. \ref{subsecPINNnumerical}.}}}
\label{fig:PINNResiduals}
\end{figure*}

\subsection{Controller Synthesis}\label{subsecPINNnumerical}
% We consider the self-assembly mechanism of colloidal particles is given by \eqref{GenSBPdyn}, where $\bm{f},\bm{g}$ are the outputs of neural network representations $\mathcal{N}_{\text {Drift}}$ and $\mathcal{N}_{\text {Diffusion}}$, respectively, and learnt from the high fidelity MD simulation data.

\begin{figure*}[htbp!]
    \centering
      \begin{subfigure}{0.8\paperwidth}
        \includegraphics[width=0.8\paperwidth]{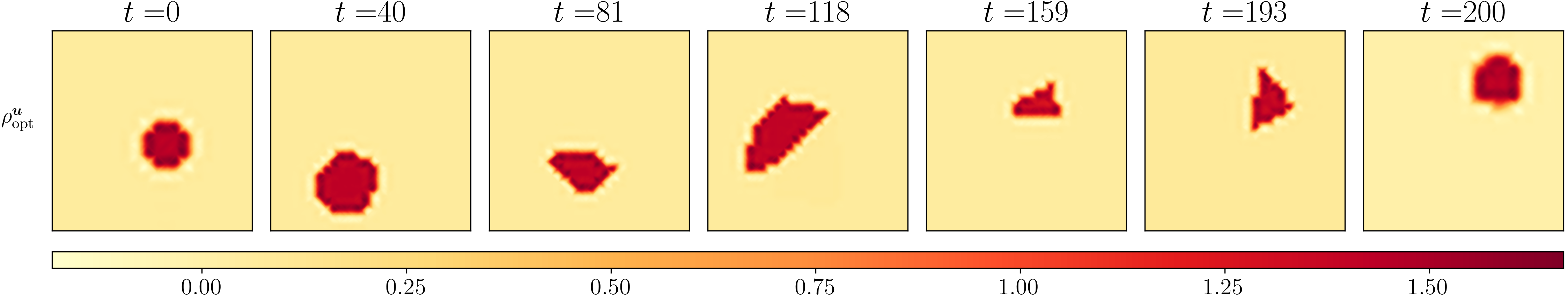}
          \caption{\small{Contour plots of the optimally controlled state PDFs $\rho^{\bm{u}}_{\rm{opt}}(t,\bm{x})$ over the state space $[0,1]^2$.}}
          \label{fig:opt_rho_opt}
      \end{subfigure}
        \hfill
      \begin{subfigure}{0.8\paperwidth}
        \includegraphics[width=0.8\paperwidth]{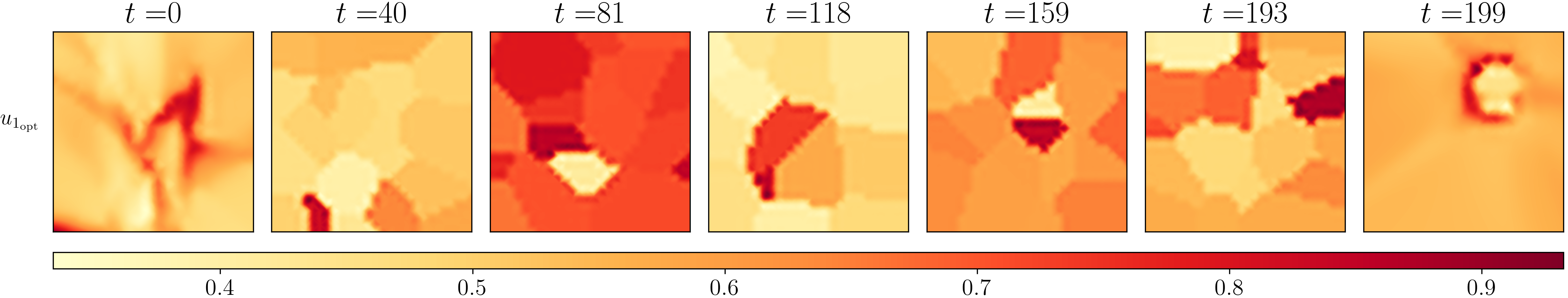}
          \caption{\small{Contour plots of the optimal control component $u_{1_{\rm{opt}}}(t,\bm{x})$ over the state space $[0,1]^2$.}}
          \label{fig:u_1}
      \end{subfigure}
      \hfill
      \begin{subfigure}{0.8\paperwidth}
        \includegraphics[width=0.8\paperwidth]{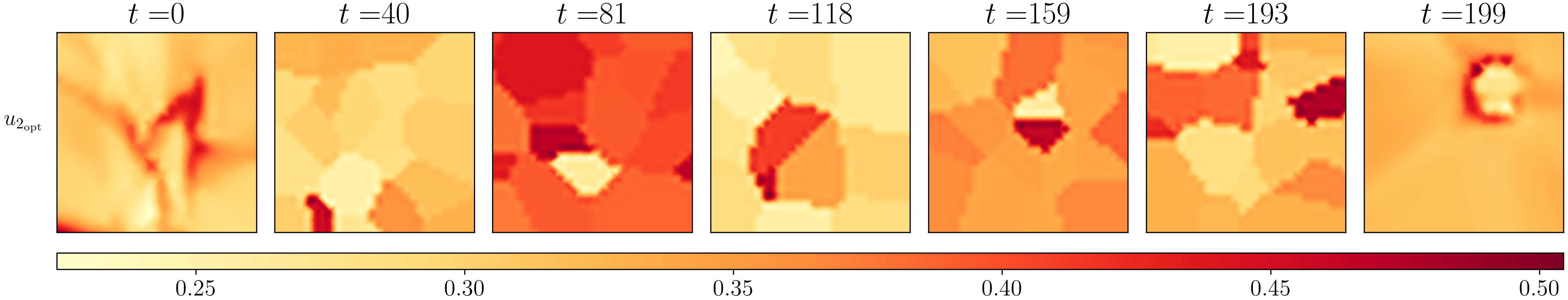}
          \caption{\small{Contour plots of the optimal control component $u_{2_{\rm{opt}}}(t,\bm{x})$ over the state space $[0,1]^2$.}}
          \label{fig:u_2}
      \end{subfigure}
            \hfill
      \begin{subfigure}{0.8\paperwidth}
        \includegraphics[width=0.8\paperwidth]{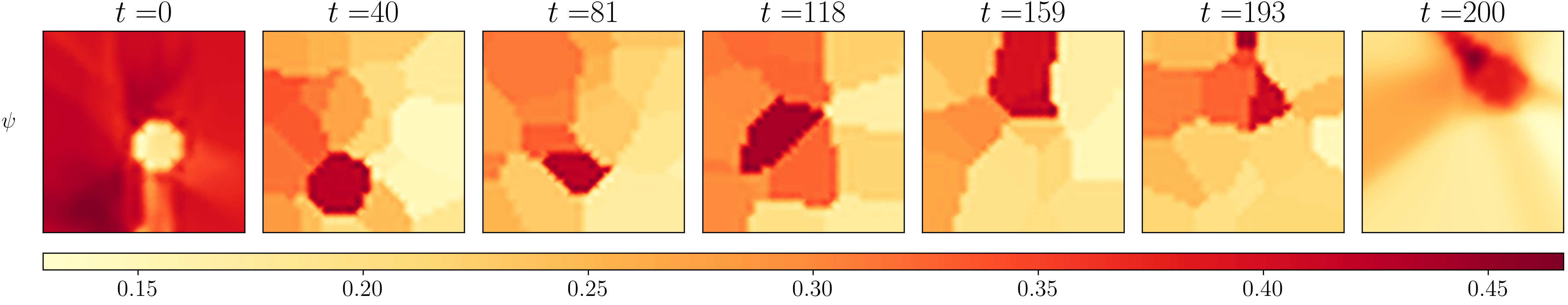}
          \caption{\small{Contour plots of the value function $\psi\left(t,\bm{x}\right)$ over the state space $[0,1]^2$.}}
          \label{fig:ValueFunction}
      \end{subfigure}
\caption{%
{\small{Results for the GSBP simulation detailed in Sec. \ref{subsecPINNnumerical} over time $t\in[0,200]$. The color denotes the value of the plotted variable; see colorbar (dark red = high, light yellow = low).}}}
\label{fig:1st_order}
\vspace*{-0.15in}
\end{figure*}
\begin{figure*}[t!]
    \centering
   \begin{subfigure}[t]{0.48\textwidth}
        \centering
        \includegraphics[height=2.8in]{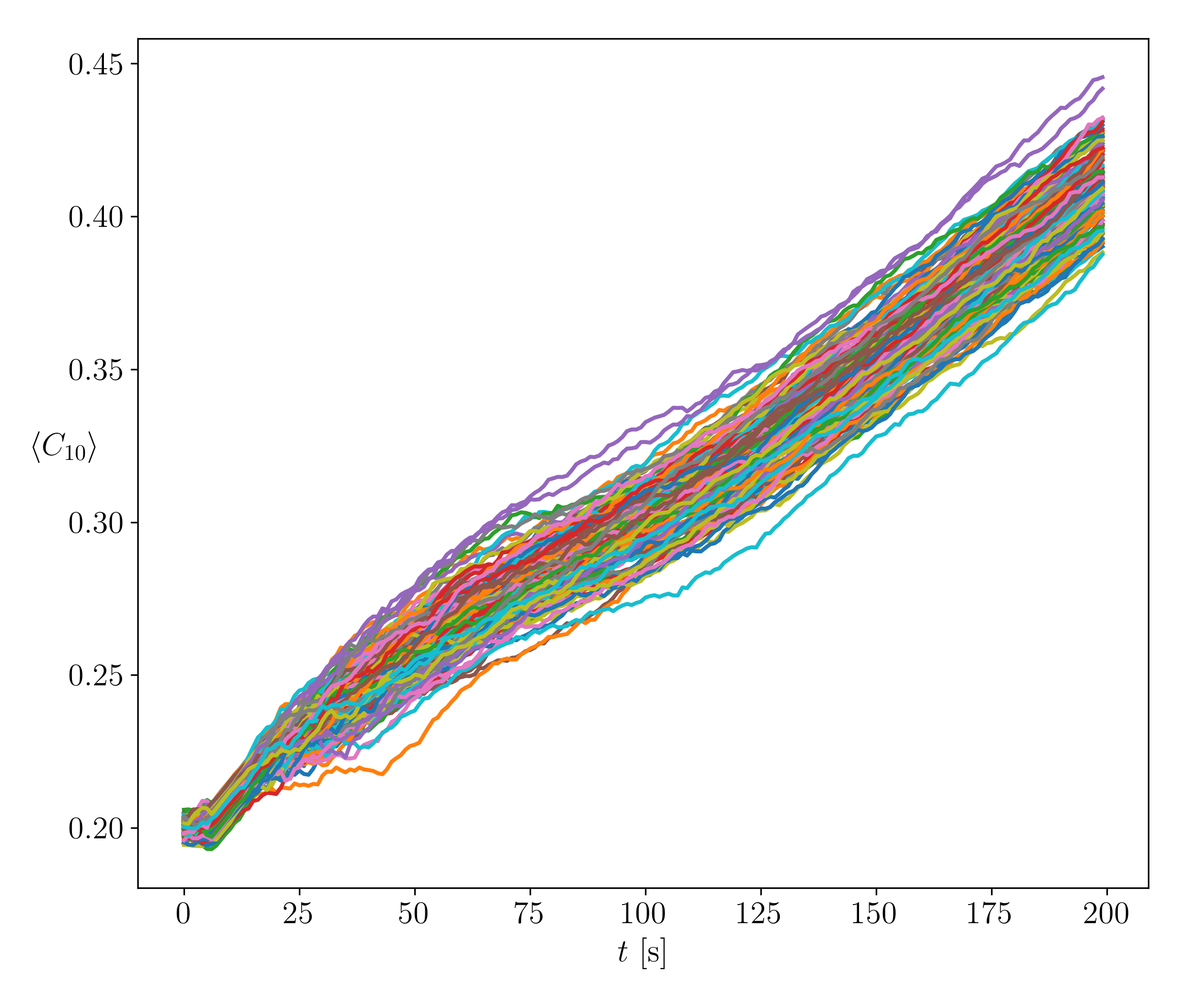}
       \caption{{\small{Optimally controlled $\langle C_{10}\rangle$ state trajectories.}}}
       \label{fig:OptimalClosed-loopStateTrajectories}
    \end{subfigure}%
    ~ 
    \begin{subfigure}[t]{0.48\textwidth}
        \centering
        \includegraphics[height=2.8in]{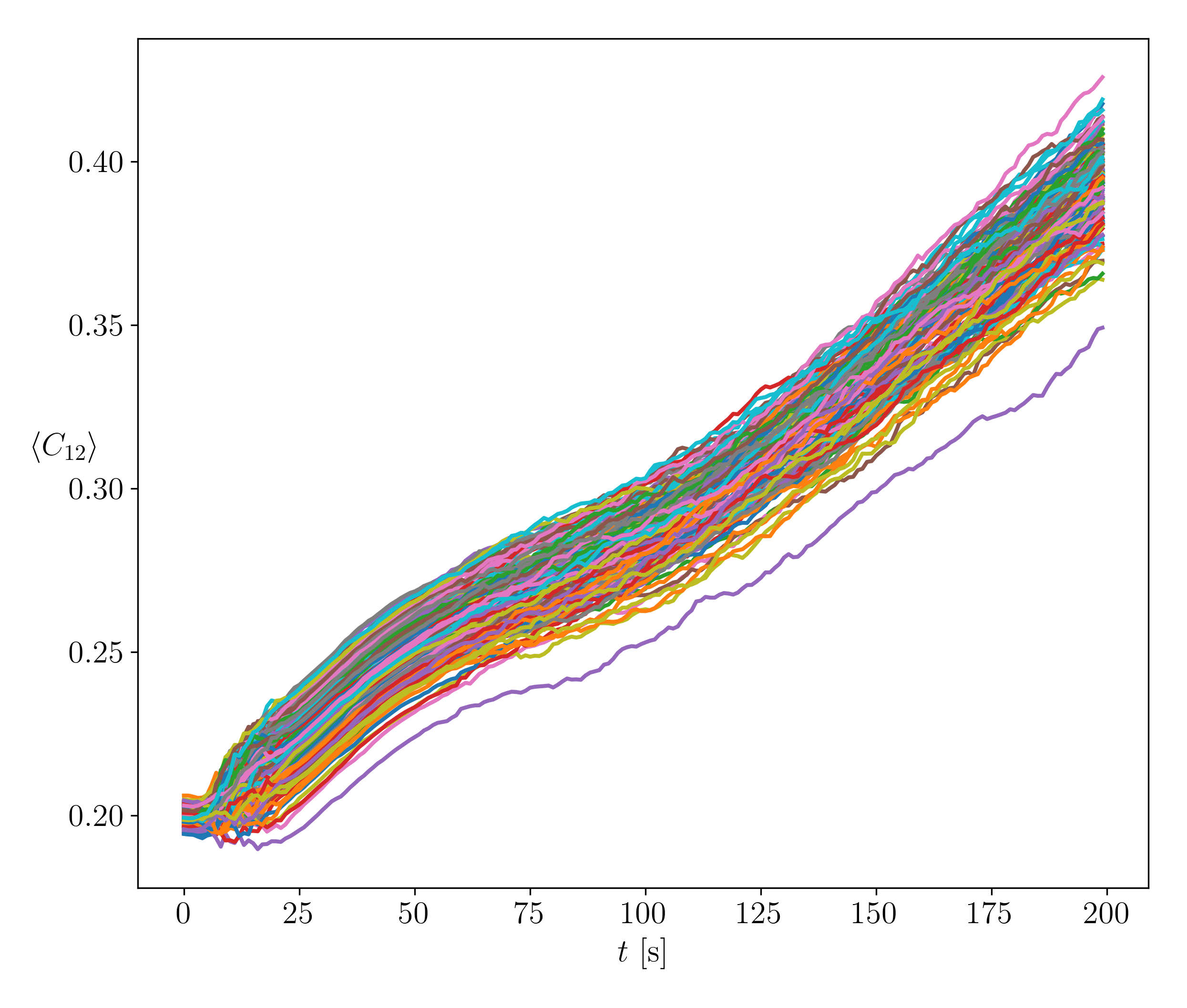}
       \caption{{\small{Optimally controlled $\langle C_{12}\rangle$ state trajectories.}}}
       \label{fig:OptimalControlTrajectories}
    \end{subfigure}
    \caption{{\small{The 150 random sample paths resulting from closed-loop simulations using the learnt optimal policy $\bm{u}_{\rm{opt}}\left(t,\langle C_{10}\rangle,\langle C_{12}\rangle\right)$.}}}
    \label{fig:SamplePath}
\end{figure*}

With the $\bm{f},\bm{g}$ learnt as per Sec. \ref{subsecLearningDyn} for the colloidal SA system described in Sec. \ref{subsecSystemSetup}, we considered the GSBP \eqref{Optimization_FPK} over fixed time horizon $[0,T]$, where the final time $T=200$ s, the initial state $\bm{x}(t=0) \sim \rho_0=\mathcal{N}\left(\bm{m}_0, \bm{\Sigma}_0\right)$, and the terminal state $\bm{x}(t=T) \sim \rho_T= \mathcal{N}\left(\bm{m}_T, \bm{\Sigma}_T\right)$. Here, the notation $\mathcal{N}\left(\bm{m},\bm{\Sigma}\right)$ stands for a joint Gaussian distribution with mean vector $\bm{m}$ and covariance matrix $\bm{\Sigma}$. We used
\begin{align}
\bm{m}_0=(0.2,0.2)^{\top},~ \bm{m}_T=(0.4,0.375)^{\top},~\bm{\Sigma}_0=\bm{\Sigma}_{T}=0.1\bm{I}_{2}.
\label{EndpointMeanCov}
\end{align}
In particular, the statistics of the initial state $\bm{x}(t=0) \sim\mathcal{N}\left(\bm{m}_0,\bm{\Sigma}_0\right)$ is chosen to coincide with that used in the MD simulation in Sec. \ref{subsecLearningDyn}. The mean $\bm{m}_T$ for the target terminal state $\bm{x}(t=T) \sim\mathcal{N}\left(\bm{m}_T,\bm{\Sigma}_T\right)$ was chosen to represent the BCC crystal structure. Hence, the control objective for the GSBP represents the problem of designing a minimum effort Markovian controller that provably steers the stochastic order parameters in a way to synthesize BCC crystal structure over the prescribed time horizon.

We used the PINN $\mathcal{N}_{\text{Schr\"{o}dinger Bridge}}$ proposed in Sec. \ref{subsecProposedPINN} for numerically solving the GSBP conditions of optimality \eqref{coupled_PDEs}-\eqref{boundary_conditions}. For our PINN implementation, the domain for state-time collocation is $\Omega= [0,1]^{2}\times [0,200]$. Our network consisted of $4$ hidden layers, each containing $70$ neurons, all with $\tanh$ activation functions. We trained the PINN for $100,000$ epochs using the Adam optimizer \cite{kingma2014adam} with a learning rate of $10^{-3}$. All our training were performed on a computing platform with NVIDIA Quadro p1000, $640$ Cuda cores, and $64$ GB RAM. For the collocation, we used $N=3000$ pseudorandom samples, drawn using Sobol distribution, between the endpoint boundary conditions at $t=0$ and $t=200$. We also uniformly randomly sampled $3,000$ samples every $20,000$ epochs to satisfy compute constraints. For computing the Sinkhorn losses at the endpoint boundary conditions, we used an entropic regularization parameter of $\varepsilon=0.1$ as in \eqref{WassDiscEntropy}. {\blue{For the computing platform mentioned above, training the proposed PINN on average takes $2$ seconds per epoch, so to complete $100,000$ epochs, it takes approximately $55.5$ hours.}}

Fig. \ref{fig:PINNResiduals} shows the PINN residuals in \eqref{totalloss}, and Fig. \ref{fig:1st_order} shows the corresponding GSBP solutions obtained from the trained PINN. In particular, Fig. \ref{fig:opt_rho_opt} shows the evolution of the optimally controlled transient joint PDFs $\rho^{\bm{u}}_{\text{opt}}(t,\bm{x})$ interpolating the fixed $\rho_{0},\rho_{T}$ mentioned above. Notice that, even though the initial and terminal stochastic states are both chosen to have Gaussian statistics, the transient joints in Fig. \ref{fig:opt_rho_opt} are non-Gaussian. This is expected since the learnt $\bm{f},\bm{g}$, as well as the optimal controller $\bm{u}_{\rm{opt}}$ (see Fig. \ref{fig:u_1}-\ref{fig:u_2}), are nonlinear in state. A comparison of Fig. \ref{fig:u_1} and Fig. \ref{fig:u_2} with Fig. \ref{fig:ValueFunction} also shows that the optimal controls are high (resp. low) in regions where the value function $\psi$ changes rapidly, i.e., when the (sub)gradient of $\psi$ is large (resp. small).

To further illustrate the GSBP results, we
performed a closed loop sample path simulation for 150 initial state samples $\bm{x}(t=0)\sim\rho_0$ (with the same $\rho_0$ mentioned before) using the learnt optimal control policy $\bm{u}_{\rm{opt}}\left(t,\langle C_{10}\rangle,\langle C_{12}\rangle\right)$ that provably steers the given $\rho_0$ from $t=0$ to the given $\rho_T$ (BCC crystal) at $t=T=200$ s. The corresponding closed-loop state sample paths shown in Fig. \ref{fig:SamplePath} demonstrate that the optimal policy indeed steers the controlled stochastic state from around $(0.2,0.2)$ to around $(0.4,0.375)$ with high probability, as specified per problem data \eqref{EndpointMeanCov}. 

{\blue{For the closed-loop simulations, we constructed a \emph{k}-d tree \cite{bentley1975multidimensional} (with leaf size = 2) for fast querying of the PINN-trained optimal control policy $\bm{u}_{\rm{opt}}\left(t,\langle C_{10}\rangle,\langle C_{12}\rangle\right)$. This construction takes $1.785$ seconds. During the numerical integration of the SDE, querying the optimal control policy takes $0.227$ milliseconds. Without the \emph{k}-d tree construction, this querying is $1000$x slower (approximately $0.22$ seconds). With \emph{k}-d tree-based querying, to simulate a closed-loop sample path as in Fig. \ref{fig:SamplePath} using the Euler-Maruyama scheme with $500$ equispaced time steps in $[0,T]$, taking approximately $177$ seconds. These experiments suggest that the proposed control approach is practically viable for colloidal SA.}}

% When we implement the control policy generated by PINN in the simulation of closed-loop trajectories, there appears to be a discrepancy between the plant model and PINN output. To address this mismatch, we introduced a bias of magnitude $10^{-3}$ to the output of both neural networks $\mathcal{N}_{\text {Drift}},\mathcal{N}_{\text {Diffusion}}$ during the execution of the closed-loop trajectories simulation.

%%%%%%%%%%%%%%%%%%%%%%%%%%%%%%%%%%%%%%%%%%%%%%%%%%%%%%%%%%%%%%%%%%%%%%%%%%%%%%%%

\section{Conclusions}\label{SecConclusions}
The work presented here proposes `neural Schr\"{o}dinger bridge' -- a novel neural network-based learning and control framework for solving the generalized Schr\"{o}dinger bridge problem (GSBP), which is a fixed time horizon stochastic optimal control problem with constraints on endpoint distributions and controlled SDEs. Our work is motivated by the problem of minimum effort controlled colloidal self-assembly (SA), where the controlled dynamics is usually not available from first principle physics and instead learnt from the MD simulation data.

Our contributions go beyond connecting the colloidal SA problem with the GSBP. Since the drift and diffusions obtained from data-driven learning of the controlled neural SDEs are nonlinear in state, non-affine in control, and have explicit time-dependence, the standard computational techniques for solving GSBPs available in the literature no longer apply in this setting. In fact, we show that the conditions of optimality for such GSBPs are very different from those studied in the literature in that here we are led to solve a system of $m+2$ coupled PDEs with two boundary conditions, where $m$ is the number of control inputs. {\blue{This system of PDEs we derive is new, and is of independent interest in the theory of Schr\"{o}dinger bridge and related stochastic optimal control problems.}} To numerically solve this non-standard system, we propose a custom variant of PINN, and demonstrate its effectiveness on a data-driven colloidal SA case study. Our results should be of broad interest to control and machine learning researchers using diffusion models for learning and control.

%%%%%%%%%%%%%%%%%%%%%%%%%%%%%%%%%%%%%%%%%%%%%%%%%%%%%%%%%%%%%%%%%%%%%%%%%%%%%%%%
% \appendix

% \subsection*{Proof for Theorem \ref{}}

%%%%%%%%%%%%%%%%%%%%%%%%%%%%%%%%%%%%%%%%%%%%%%%%%%%%%%%%%%%%%%%%%%%%%%%%%%%%%%%%

\bibliographystyle{IEEEtran} 
\bibliography{TCST_SBP_SA.bib}

\end{document}